\documentclass[reqno,12pt]{amsart}
\usepackage{amssymb,amsmath,amsthm,amsxtra,calc,bm, color}
\usepackage{enumerate}
\usepackage[margin=.95in]{geometry}

\usepackage[scr=boondoxo]{mathalfa}
\usepackage{mathtools}

\usepackage{mleftright}
\mleftright

\usepackage{nccmath}
\usepackage{cases}

\usepackage{calc}
\usepackage{graphicx}

\usepackage[hidelinks]{hyperref}

\usepackage[final]{microtype}

\def\Z{\mathbb{Z}}

\def\R{\mathbb{R}}
\def\H{\mathcal{H}}

\def\C{\mathbb{C}}

\def\K{\mathbb{K}}
\def\Cl{\mathrm{Cl}^+}

\def\sQ{\mathcal{Q}}
\def\sF{\mathcal{F}}
\def\sN{\mathcal{N}}

\DeclareMathOperator{\ch}{cosh}
\DeclareMathOperator{\im}{Im}
\DeclareMathOperator{\re}{Re}
\DeclareMathOperator{\Tr}{Tr}

\DeclareMathOperator{\Ci}{Ci}
\DeclareMathOperator{\Si}{Si}
\def\SL{{\rm SL}}
\def\PSL{{\rm PSL}}

\newcommand{\pfrac}[2]{\left(\frac{#1}{#2}\right)}
\newcommand{\pmfrac}[2]{\left(\mfrac{#1}{#2}\right)}
\newcommand{\ptfrac}[2]{\left(\tfrac{#1}{#2}\right)}
\newcommand{\pMatrix}[4]{\left(\begin{matrix}#1 & #2 \\ #3 & #4\end{matrix}\right)}
\newcommand{\ppMatrix}[4]{\left(\!\pMatrix{#1}{#2}{#3}{#4}\!\right)}
\renewcommand{\pmatrix}[4]{\left(\begin{smallmatrix}#1 & #2 \\ #3 & #4\end{smallmatrix}\right)}
\renewcommand{\bar}[1]{\overline{#1}}

\renewcommand{\a}{\mathfrak{a}}
\renewcommand{\b}{\mathfrak{b}}
\renewcommand{\c}{\mathfrak{c}}
\newcommand{\fS}{\mathfrak{S}}

\renewcommand{\hat}{\widehat}
\renewcommand{\tilde}{\widetilde}

\renewcommand{\sl}{\big|}

\DeclareMathOperator{\sgn}{sgn}

\def\ep{\varepsilon}

\newtheorem{theorem}{Theorem}[section]
\newtheorem{lemma}[theorem]{Lemma}

\newtheorem{proposition}[theorem]{Proposition}

\theoremstyle{remark}
\newtheorem*{remark}{Remark}
\newtheorem*{remarks}{Remarks}

\numberwithin{equation}{section}

\mathtoolsset{showonlyrefs}

\title[Modular invariants for real quadratic fields]{Modular invariants for real quadratic fields and Kloosterman sums}

\date{\today}

\author{Nickolas Andersen}
\address{Brigham Young University Department of Mathematics,
Provo, UT 84602}
\email{nick@math.byu.edu}
\author{William Duke}
\address{UCLA Mathematics Department,
Los Angeles, CA 90095} 
\email{wdduke@ucla.edu}
 
\thanks{The authors were supported by NSF grant DMS--1701638. The second author was also suppored by the Simons Foundation, award number 554649.}

\begin{document}

\begin{abstract}
We investigate the asymptotic distribution of integrals of the $j$-function that are associated to ideal classes in a real quadratic field. 
To estimate the error term in our asymptotic formula, we prove a bound for sums of Kloosterman sums of half-integral weight that is uniform in every parameter. 
To establish this estimate we prove a variant of Kuznetsov's formula where the spectral data is restricted to half-integral weight forms in the Kohnen plus space, and we apply Young's hybrid subconvexity estimates for twisted modular $L$-functions.
\end{abstract}

\maketitle

\section{Introduction}

The relationship between modular forms and quadratic fields is exceedingly rich.
For instance, the Hilbert class field of an imaginary quadratic field may be generated by adjoining to the quadratic field  a special value of the modular $j$-function.
The connection between  class fields  of real quadratic fields and invariants of the modular group is much less  understood,
although there has been striking progress lately by Darmon and Vonk  \cite{DV}.
Our  aim in this paper is to study the asymptotic behavior of certain integrals of the modular $j$-function that are associated to ideal classes in a real quadratic field.
Before turning to this,  it is 
useful to make some definitions and to recall the corresponding problem in the imaginary quadratic case.

Let $\K$ be the quadratic field of discriminant $d$ and 
let $\Cl_d$ denote the narrow class group of $\K$.
Let $h(d) = \#\Cl_d$ denote the class number.
If $d<0$
then each ideal class $A\in \Cl_d$ contains exactly one fractional ideal of the form $z_A\Z+\Z$, where
\begin{equation}
	z_A=\frac{-b+i\sqrt {|d|}}{2a}
\end{equation}
for some relatively prime integers $a,b,c$ with $a>0$ and $b^2-4ac=d$, and where $z_A$ is in the fundamental domain
\begin{equation}
	\sF := \{z\in \H : -\tfrac 12< \re z\leq \tfrac 12, |z|\geq 1\}
\end{equation}
for the action of the modular group $\Gamma_1=\PSL_2(\Z)$.
Such $z_A$ are called \emph{reduced}.
A beautiful result from the theory of complex multiplication states that the values $j_1(z_A)$, as $A$ runs over ideal classes of discriminant $d$, are conjugate algebraic integers.
Here $j_1=j-744$ is the normalized modular $j$-invariant
\begin{equation}
	j_1(z) = q^{-1} + 196\,884 q + 21\,493\,760 q^2 + \ldots,
\end{equation}
where $q=e(z)=e^{2\pi iz}$.
It follows that the trace
\begin{equation} \label{eq:neg-trace}
	\Tr_d(j_1) := \frac{1}{\omega_d}\sum_{A\in \Cl_d} j_1(z_A),
\end{equation}
where $\omega_{-3}=3$, $\omega_{-4}=2$, and $\omega_d=1$ otherwise,
is a rational integer.
For example,
\begin{equation}
	\Tr_{-3}(j_1) = -248, \quad \Tr_{-4}(j_1) = 492, \quad \Tr_{-7}(j_1) = -4119, \quad \Tr_{-8}(j_1) = 7256.
\end{equation}

It is natural to ask how these values are distributed as $|d|\to\infty$. 
As a first approximation, it is not too hard to show that $\Tr_d(j_1)\sim (-1)^d \exp(\pi \sqrt{|d|})$ for large $d$,
but in fact much more is known.
In \cite{bjo} it was observed, and in \cite{duke-24} the second author proved, that
\begin{equation} \label{eq:24}
	\Tr_d(j_1) \, - \!\sum_{\im z_A>1}\! e(-z_A) \sim -24 h(d) 
\end{equation}
as $d\to -\infty$ through fundamental discriminants.
The value $-24$ is a suitably defined ``average of $j_1$'' over the fundamental domain $\sF$ (see \cite{duke-24}).

Now suppose that $\K$ is a real quadratic field, i.e. $d>0$.
Each ideal class $A\in \Cl_d$ contains a fractional ideal of the form  $w\Z+\Z\in A$ where $w \in \K$ is such that
\begin{equation} \label{eq:reduced-def}
	0 < w^\sigma < 1 < w,
\end{equation}
where $\sigma$ is the nontrivial Galois automorphism of $\K$. 
Such $w$ are called \emph{reduced} (in the sense of Zagier \cite{Za}); unlike in the imaginary quadratic case, a given ideal class may have many reduced representatives.
Let $\mathcal S_w$ be the oriented hyperbolic geodesic in $\H$ from $w^\sigma$ to $w$, and let $\mathcal C_A$ be the closed geodesic obtained by projecting $\mathcal S_w$ to $\Gamma_1\backslash \H$.
The choice of reduced $w$ does not affect $\mathcal C_A$.
One can view $\mathcal C_A$ in $\H$ as the geodesic from some point $z_0$ on $\mathcal S_w$ to $\gamma_w (z_0)$, where $\gamma_w$ is the hyperbolic element which generates the stabilizer of $w$ in $\Gamma_1$.
It is well-known that
\begin{equation}
	\operatorname{length}(\mathcal C_A) = 2\log \ep_d,
\end{equation}
where $\ep_d$ is the fundamental unit of $\K$.

A real quadratic analogue of the trace \eqref{eq:neg-trace} is the sum of integrals
\begin{equation} \label{eq:pos-trace}
	\Tr_d(j_1) := \sum_{A\in \Cl_d} \int_{\mathcal C_A} j_1(z) \mfrac{|dz|}{y},
\end{equation}
and one might ask how these invariants are distributed as the discriminant $d$ varies.
Numerically, we have
\begin{equation}
	\Tr_5(j_1)\approx-11.5417, \quad \Tr_8(j_1)\approx-19.1374, \quad \Tr_{13}(j_1)\approx-23.4094, \quad \Tr_{17}(j_1)\approx-43.9449.
\end{equation}
Note that these values are quite small even though $j_1$ grows exponentially in the cusp.
It was conjectured in \cite{DIT-cycle} that
\begin{equation} \label{eq:pos-trace-asymp}
	\Tr_d(j_1) \sim -24 \cdot 2\log \ep_d \, h(d)
\end{equation}
as $d\to\infty$ through fundamental discriminants.
This was proved independently in \cite{dfi-weyl} (for odd fundamental discriminants, with a power-saving of $d^{-\frac 1{5325}}$) and in \cite{masri-cycle} (for all fundamental discriminants, with a power-saving of $d^{-\frac1{400}}$).

The real quadratic invariants $\Tr_d(j_1)$ were first studied in \cite{DIT-cycle} in the context of harmonic Maass forms (nonholomorphic modular forms which are annihilated by the hyperbolic Laplacian).
There is a family of harmonic Maass forms $\{f_{d'}\}$ of weight $\frac 12$, indexed by positive discriminants $d'$, whose Fourier coefficients can be written in terms of the sums \eqref{eq:pos-trace} twisted by genus characters.
For each factorization $D=dd'$ of the fundamental discriminant $D$ into fundamental discriminants $d,d'$, there is a real character $\chi_d=\chi_{d'}$ of $\Cl_D$ called a genus character.
The $d$-th Fourier coefficient of $f_{d'}$ is given by
\begin{equation}
	\Tr_{d,d'}(j_1) := \sum_{A\in \Cl_D} \chi_d(A) \int_{\mathcal C_A} j_1(z) \mfrac{|dz|}{y}.
\end{equation}
In particular, the $d$-th Fourier coefficient of $f_1$ is $\Tr_d(j_1)$.
The remaining non-square-indexed coefficients can be described in terms of $\Tr_{d,d'}(j_m)$ for $m\geq 1$, where
$j_m$ is the unique modular function in $\C[j]$ of the form $j_m=q^{-m} + O(q)$.
Our first result concerns the asymptotic distribution of the values of $\Tr_{d,d'}(j_m)$ as any of the parameters $d,d',m$ tends to infinity.
We define $\delta_1 = 1$ and $\delta_d=0$ otherwise, and $\sigma_s(n) = \sum_{\ell\mid n} \ell^s$ for any $s\in \C$.

\begin{theorem} \label{thm:main-thm-cycle-int}
For each positive fundamental discriminant $D$, let $d$ be any positive fundamental discriminant dividing $D$.
Then for each $m\geq 1$ we have
\begin{equation} \label{eq:main-thm-cycle-int}
	\sum_{A\in \Cl_D} \chi_d(A) \int_{\mathcal C_A} j_m(z) \mfrac{|dz|}{y} = -24 \, \delta_{d} \, \sigma_1(m) \cdot 2h(D) \log \ep_D  + O\left( m^{\frac {8}{9}}D^{\frac{13}{27}}(mD)^\ep \right).
\end{equation}
\end{theorem}

\begin{remarks}
In the case $d=1$, the power-saving of $D^{-\frac 1{54}}$ in Theorem~\ref{thm:main-thm-cycle-int} improves on the results of \cite{masri-cycle,dfi-weyl}.
The generalizations to $d>1$ and $m>1$ are new, and the latter confirms the observation in \cite{DIT-cycle} that $\Tr_D(j_m) \sim -24\sigma_1(m)\cdot 2\log \ep_D h(D)$ as $m\to\infty$.

When $D=dd'$ is a factorization of $D$ into negative fundamental discriminants, the left-hand side of \eqref{eq:main-thm-cycle-int} is identically zero.
To see this, let $J$ denote the class of the different $(\sqrt D)$ of $\K$.
The closed geodesic associated to $JA^{-1}$ has the same image in $\Gamma_1\backslash \H$ as $\mathcal C_A$ but with the opposite orientation.
Since $\chi_d(J) = \sgn d$, the left-hand side of \eqref{eq:main-thm-cycle-int} is forced to vanish whenever $d<0$.
\end{remarks}

In order to give a better geometric interpretation when $D=dd'$ where $d$ and $d'$ are negative,
Imamo\=glu, T\'oth, and the second author \cite{DIT-geom} recently defined a new invariant $\sF_A$, which is a finite area hyperbolic surface with boundary $\mathcal C_A$.
We briefly describe the construction of $\sF_A$; for details see \cite{DIT-geom}.
Let $w$ be one of the reduced quadratic irrationalities associated to $A$, and let $\gamma_w\in \Gamma_1$ be the hyperbolic element that fixes $w$ and $w^\sigma$.
Then $\gamma_w$ can be written as
\begin{equation} \label{eq:gamma-w-ts}
	\gamma_w = T^{\lceil w\rceil}ST^{n_1}ST^{n_2}S\cdots T^{n_{\ell}}S 
\end{equation}
for some integers $n_i\geq 2$, where $T=\pm\pmatrix 1101$ and $S=\pm\pmatrix 0{-1}10$ are generators of $\Gamma_1$.
The cycle $(n_1,\ldots n_\ell)$ is the period of the minus continued fraction of $w$, and $\ell$ is the number of distinct reduced representatives of $A$.
Let $S_k:= T^{(n_1+\cdots +n_k)}ST^{-(n_1+\cdots +n_k)}$ and define
\begin{equation}
	\Gamma_A := \langle S_1, \ldots, S_\ell, T^{(n_1+\cdots+n_\ell)}\rangle.
\end{equation}
This group is an infinite-index (i.e.~thin) subgroup of $\Gamma_1$.
Let $\sN_A$ be the Nielsen region of $\Gamma_A$: 
the smallest non-empty $\Gamma_A$-invariant open convex subset of $\H$.
Then the surface $\sF_A$ is defined as $\Gamma_A\backslash \sN_A$.
A different choice of reduced $w$ representing $A$ yields a subgroup of $\Gamma_1$ conjugate to $\Gamma_A$ by a translation, so the surface $\sF_A$ is uniquely defined by $A$; see Theorem~1 of \cite{DIT-geom}.
In that theorem we also find that the area of $\sF_A$ is $\pi \ell$, with $\ell$ as in \eqref{eq:gamma-w-ts}.

Our second result concerns the distribution of sums of the integrals of $j_m$ over the surfaces $\sF_A$ as the discriminant varies.
The functions $j_m$ grow exponentially in the cusp, so we regularize\footnote{The published version of this paper uses a different regularization that we erroneously claimed was equivalent to the one given here using $\nu_A(z)$. See Section~\ref{sec:geom-inv} for an explanation of how these two regularizations differ.} the integrals using the function $\nu_A(z)$ defined in Section~\ref{sec:geom-inv}.

\begin{theorem} \label{thm:main-thm-surf}
For each positive fundamental discriminant $D$, let $D=dd'$ be any factorization into negative fundamental discriminants.
Then for each $m\geq 1$ we have
\begin{equation} \label{eq:main-thm-surf}
	\mfrac{1}{4\pi}\sum_{A\in \Cl_D} \chi_d(A) \int_{\sF} j_m(z) \nu_A(z) \mfrac{dxdy}{y^2} = -24 \sigma_1(m) \frac{h(d)h(d')}{\omega_d \omega_{d'}} + O\left( m^{\frac {8}{9}}D^{\frac{13}{27}}(mD)^\ep \right).
\end{equation}
\end{theorem}

\begin{remark}
When $D=dd'$ is a factorization into positive discriminants, the left-hand side of \eqref{eq:main-thm-surf} is identically zero because $A\mapsto JA^{-1}$ reverses the orientation of the surface $\mathcal F_A$.
\end{remark}

An interesting special case occurs when $D=4p$ where $p\equiv 3\pmod{4}$ is a prime.
In this case the identity class $I=I_p$ is not equivalent to the class of the different $J=J_p$.
The Cohen-Lenstra heuristics predict that approximately $75\%$ of such fields have wide class number one, which would imply that the classes containing $I$ and $J$ are the only ideal classes.
If this is the case, then there is a sequence of primes $p\equiv 3\pmod{4}$ for which
\begin{equation}
	\int_{\sF_{I_p}} j_1(z) \mfrac{dxdy}{y^2} \sim -2\pi h(-p) \quad \text{ and } \quad \int_{\sF_{J_p}} j_1(z) \mfrac{dxdy}{y^2} \sim 2\pi h(-p).
\end{equation}

The method used in \cite{duke-24} to prove \eqref{eq:24} and in \cite{masri-cycle} to prove \eqref{eq:pos-trace-asymp} involves the equidistribution of CM points and closed geodesics originally developed by the second author in \cite{duke-half-integral}.
By contrast, here we employ a relation between the invariants in \eqref{eq:main-thm-cycle-int} and \eqref{eq:main-thm-surf} and sums of Kloosterman sums (see Section~\ref{sec:geom-inv}).
We then estimate the sums of Kloosterman sums directly via a Kuznetsov-type formula.

The Kloosterman sums in question are those which appear in the Fourier coefficients of Poincar\'e series of half-integral weight in the Kohnen plus space.
In weight $k=\lambda+\frac 12$, the plus space consists of holomorphic or Maass cusp forms whose Fourier coefficients are supported on exponents $n$ such that $(-1)^\lambda n\equiv 0,1\pmod{4}$.
For integers $m,n$ satisfying the plus space condition and $c$ a positive integer divisible by $4$ we define
\begin{equation} \label{eq:sk+def}
	S_k^+(m,n,c) := e\Big(\!-\mfrac k4\Big) \sum_{d\bmod c} \pfrac{c}{d} \ep_d^{2k} e\pfrac{m\bar d+n d}{c} \times
	\begin{cases}
		1 & \text{ if } 8\mid c,\\
		2 & \text{ if }4\mid\mid c,
	\end{cases}
\end{equation}
where $d\bar d\equiv 1\pmod{c}$ and $\ep_d=1$ or $i$ according to $d\equiv 1$ or $3\pmod{4}$, respectively.
The Kloosterman sums \eqref{eq:sk+def} are real-valued and satisfy the relation
\begin{equation} \label{eq:kloo-neg-k}
	S_k^+(m,n,c) = S_{-k}^+(-m,-n,c).
\end{equation}
We prove a strong uniform bound for these sums which is of independent interest.
We remark that similar (but weaker) estimates are hiding in the background of the methods of \cite{duke-24,masri-cycle}.

\begin{theorem} \label{thm:main-kloo}
Let $k=\pm\frac 12=\lambda+\frac 12$.
Suppose that $m,n$ are positive integers such that $(-1)^\lambda m=v^2 d'$ and $(-1)^\lambda n=w^2d$, where $d,d'$ are fundamental discriminants not both equal to $1$.
Then
\begin{equation} \label{eq:main-kloo}
 	\sum_{4\mid c\leq x} \frac{S_{k}^+(m,n,c)}{c} \ll \left(x^{\frac 16} + (dd')^{\frac 29}(vw)^{\frac 13}\right)(mnx)^\varepsilon.
\end{equation} 
\end{theorem}

Friedlander, Iwaniec, and the second author \cite{dfi-weyl} proved an analogous estimate for smoothed sums of Kloosterman sums on $\Gamma_0(4q)$ with a power saving of $n^{-1/1330}$ when $n$ is squarefree.
Individually, the Kloosterman sums satisfy the Weil-type bound
\begin{equation} \label{eq:weil-bound}
	|S_k^+(m,n,c)| \leq 2 \sigma_0(c) \gcd(m,n,c)^{\frac 12} \sqrt{c},
\end{equation}
(see, e.g., Lemma~6.1 of \cite{dfi-weyl})
so the sum in \eqref{eq:main-kloo} is trivially bounded above by $(mnx)^{\varepsilon} \sqrt x$.
Theorem~\ref{thm:main-kloo} should be compared with the bound of Sarnak and Tsimerman \cite{sarnak-tsimerman} for the ordinary integral weight Kloosterman sums $S(m,n,c)$ which improves on the pivotal result of Kuznetsov in \cite{kuznetsov}.
The main result of \cite{sarnak-tsimerman} is unconditional and depends on progress toward the Ramanujan conjecture for Maass cusp forms of weight $0$.
Assuming that conjecture, their theorem states that
\begin{equation}
	\sum_{c\leq x} \frac{S(m,n,c)}{c} \ll \left(x^{\frac 16} + (mn)^{\frac 16}\right)(mnx)^\varepsilon.
\end{equation}
Our method also yields an exponent of $\frac 16$ for $dd'$ in \eqref{eq:main-kloo} if we assume the Lindel\"of hypothesis for $L(\frac 12,\chi)$ and $L(\frac 12,f\times\chi)$, where $\chi$ is a quadratic Dirichlet character and $f$ is an integral weight cusp form (holomorphic or Maass).
Via the correspondence of Waldspurger, the Lindel\"of hypothesis for all such $L(\frac 12,f\times \chi)$ is equivalent to the Ramanujan conjecture for half-integral weight forms.

Recently Ahlgren and the first author used a similar approach to study the half-integral weight Kloosterman sums associated to the multiplier system for the Dedekind eta function.
This was used in \cite{aa-kloosterman} to improve the error bounds of \cite{lehmer-remainders,folsom-masri} for the classical formula of Hardy, Ramanujan, and Rademacher for the partition function $p(n)$.
In particular, it was shown that the discrepancy between $p(n)$ and the first $O(\sqrt n)$ terms in the formula is at most $O(n^{-\frac 12-\frac{1}{168}+\varepsilon})$.

The proof of Theorem~\ref{thm:main-kloo} hinges on a version of Kuznetsov's formula which relates the Kloosterman sums \eqref{eq:sk+def} to the coefficients of holomorphic cusp forms, Maass cusp forms, and Eisenstein series of half-integral weight in the plus space.
One advantage of the plus space is that the Waldspurger correspondence is completely explicit on that space via \cite{kohnen-zagier} and \cite{baruch-mao}; knowledge of the exact proportionality constant in the Waldspurger correspondence is crucial for us.
Here we briefly define the relevant quantities and state a special case of our version of the Kuznetsov formula.
Let $\mathcal H_k^+$ (resp.~$\mathcal V_k^+$) denote an orthonormal Hecke basis for the plus space of holomorphic (resp.~Maass) cusp forms of weight $k$ for $\Gamma_0(4)$.
For each $g\in \mathcal H_k^+$ (resp.~$u_j\in \mathcal V_k^+$) let $\rho_g(n)$ (resp.~$\rho_j(n)$) denote the suitably normalized $n$-th Fourier coefficient of $g$ (resp.~$u_j$).
For each $j$, let $\lambda_j = \frac 14+r_j^2$ denote the Laplace eigenvalue of $u_j$.
The full statement with detailed definitions appears in Section~\ref{sec:plus} below.

\begin{theorem} \label{thm:thm-kuz-intro}
Let $k=\pm \frac12=\lambda+\frac 12$.
Suppose that $m,n$ are positive integers such that $(-1)^\lambda m$ and $(-1)^\lambda n$ are fundamental discriminants.
Suppose that $\varphi:[0,\infty)\to\R$ is a smooth test function which satisfies \eqref{eq:varphi-cond}, and let $\widetilde\varphi$ and $\widehat\varphi$ denote the integral transforms in \eqref{eq:phi-tilde-def}--\eqref{eq:phi-hat-def}. 
Then
\begin{multline} \label{eq:kuz-thm-intro}
\sum_{4\mid c>0} \frac{S_k^+(m,n,c)}{c} \varphi\pfrac{4\pi\sqrt{mn}}{c} \\ =
6\sqrt{mn} \sum_{u_j \in \mathcal V_k^+} \frac{\overline{\rho_{j}(m)}\rho_{j}(n)}{\cosh\pi r_j} \widehat\varphi(r_j) + 
 \mfrac 32 \sum_{\ell\equiv k\bmod 2} e\ptfrac{\ell-k}4 \widetilde\varphi(\ell) \Gamma(\ell) \sum_{g\in \mathcal H_\ell^+} \overline{\rho_{g}(m)}\rho_{g}(n)  \\
+ \int_{-\infty}^\infty \pfrac{n}{m}^{ir} \frac{L(\frac 12-2ir,\chi_{(-1)^\lambda m}) L(\frac 12+2ir,\chi_{(-1)^\lambda n})}{2\cosh\pi r |\Gamma(\frac{k+1}{2}+ir)|^2 |\zeta(1+4ir)|^2} \widehat\varphi(r) \, dr.
\end{multline}
\end{theorem}

\begin{remark}
This version of the Kuznetsov formula for Maass forms in the plus space for $\Gamma_0(4)$  with weight $\pm \frac 12$  is precisely analogous to the original version of Kuznetsov's formula for the full modular group.
To prove it
we apply Bir\'o's idea \cite{biro} of taking a linear combination of Proskurin's Kuznetsov-type formula evaluated at various cusp-pairs in order to project the holomorphic and Maass cusp forms into the plus space.
The main technical complication arises from the sum of Eisenstein series terms from Proskurin's formula, which we show simplifies to the integral of Dirichlet $L$-functions in \eqref{eq:kuz-thm-intro}.
The simplicity of that integral is reminiscent of the corresponding term in Kuznetsov's original formula \cite[Theorem~2]{kuznetsov} for the ordinary weight $0$ Kloosterman sums;
in that formula, the Eisenstein term is
\begin{equation}
	\frac 1\pi\int_{-\infty}^{\infty} \pmfrac nm^{ir} \frac{\sigma_{2ir}(m)\sigma_{-2ir}(n)}{|\zeta(1+2ir)|^2} \hat\varphi(r) \, dr.
\end{equation}
Note that if $k=0$ then $\cosh\pi r|\Gamma(\frac {k+1}2+ir)|^2=\pi$.
\end{remark}

The most crucial input in the proof of Theorem~\ref{thm:main-kloo} is Young's Weyl-type hybrid subconvexity estimates \cite{young-subconvexity} for $L(\frac 12,f\times\chi_d)$ and $L(\frac 12,\chi_d)$ which improve on the groundbreaking results of Conrey and Iwaniec \cite{conrey-iwaniec}.
Young proved that
\begin{equation} \label{eq:young-thm-intro}
	\sum_{f} L(\tfrac 12,f\times \chi_d)^3 \ll (kd)^{1+\varepsilon}
\end{equation}
for odd fundamental discriminants $d$, where
the sum is over all holomorphic newforms of weight $k$ and level dividing $d$.
In Appendix~\ref{sec:appendix} we sketch the details required to generalize Young's result to twists by $\chi_d$ for even fundamental discriminants $d$, where the sum is over $f$ of level dividing the squarefree part of $d$.
The uniformity of Young's result in both the level and weight directly influences the quality of the exponents in \eqref{eq:main-kloo}.
There are corresponding results in \cite{young-subconvexity} for twisted $L$-functions of Maass cusp forms and Dirichlet $L$-functions which we also use in the proof of Theorem~\ref{thm:main-kloo}.

\begin{remark}
The condition in \eqref{eq:young-thm-intro} (and our extension in Appendix~\ref{sec:appendix}) that $f$ have level dividing the squarefree part of $d$ (which is odd unless $d=4q$ with $q\equiv 2\pmod{4}$)
is why we require a Kuznetsov formula that involves only coefficients of cusp forms in the plus space.
Under the Shimura correspondence, the plus spaces of half-integral weight forms on $\Gamma_0(4)$ are isomorphic as Hecke modules to spaces of weight $0$ cusp forms on $\Gamma_0(1)$, whereas
the full spaces on $\Gamma_0(4)$ lift to $\Gamma_0(2)$.
\end{remark}

The paper is organized as follows.
In Section~\ref{sec:geom-inv} we use the formulas of \cite{DIT-geom} to relate the geometric invariants to sums of Kloosterman sums, and we apply Theorem~\ref{thm:main-kloo} to prove Theorems~\ref{thm:main-thm-cycle-int} and \ref{thm:main-thm-surf}.
The remainder of the paper is dedicated to the proof of Theorem~\ref{thm:main-kloo}.
In Section~\ref{sec:background} we give some background on the spectrum of the hyperbolic Laplacian in half-integral weight.
In Section~\ref{sec:MVE} we prove general estimates for the mean square of Fourier coefficients of Maass cusp forms of half-integral weight with arbitrary multiplier system.
We prove Theorem~\ref{thm:thm-kuz-intro} in Section~\ref{sec:plus} and Theorem~\ref{thm:main-kloo} in Section~\ref{sec:proof}.
Finally, Appendix~\ref{sec:appendix} contains a sketch of the proof of Young's subconvexity result extended to even discriminants.

\subsection*{Acknowledgement} The authors thank the referee for their thorough and careful reading of an earlier version of the manuscript, as well as their helpful comments and suggestions. We are also indebted to Vaibhav Kalia and Balesh Kumar, who discovered some mistakes in the published version of this paper, and who carefully corrected these mistakes in their work \cite{kalia-kumar} (see the three footnotes in Sections~1 and 2 of this paper for details and corrections).

\section{Geometric invariants and Kloosterman sums} \label{sec:geom-inv}

In this section we relate the real quadratic invariants to Kloosterman sums and show how Theorems~\ref{thm:main-thm-cycle-int} and \ref{thm:main-thm-surf} follow from Theorem~\ref{thm:main-kloo}.
Actually, we will prove more general forms of the main theorems which allow for non-fundamental discriminants.
It is convenient to use binary quadratic forms
\begin{equation}
	Q(x,y) = [a,b,c] = ax^2+bxy+cy^2
\end{equation}
in place of ideal classes, as this point of view makes the generalization to arbitrary discriminants straightforward.
A discriminant is any integer $D\equiv 0,1\pmod{4}$.
A discriminant $D$ is fundamental if it is either odd and squarefree or if $D/4$ is squarefree and congruent to $2,3\pmod{4}$.
Fix a discriminant $D>1$ and a factorization $D=dd'$ into positive or negative discriminants $d,d'$ such that $d$ is fundamental.
Let $\sQ_D$ be the set of all integral binary quadratic forms $[a,b,c]$ with discriminant $b^2-4ac=D$.
The modular group $\Gamma_1$ acts on $\sQ_D$ in the usual way.
When $D$ is fundamental all forms in $\sQ_D$ are primitive (i.e.~$\gcd(a,b,c)=1$) and there is a simple correspondence between $\Gamma_1\backslash \sQ_D$ and $\Cl_D$ via
\begin{equation} \label{eq:quad-form-ideal}
	[a,b,c] \mapsto w\Z+\Z, \quad \text{ where } w = \mfrac{-b+\sqrt{D}}{2a},
\end{equation}
assuming $[a,b,c]$ is chosen in its class to have $a>0$.
If $D$ is fundamental and if $Q$ corresponds to $A$ via \eqref{eq:quad-form-ideal} then we define $\mathcal C_Q := \mathcal C_A$ and $\sF_Q := \sF_A$.
We extend this to arbitrary discriminants via $\mathcal C_{\delta Q} := \mathcal C_Q$ and $\sF_{\delta Q} := \sF_Q$.
There is a generalized genus character $\chi_d$ on $\Gamma_1\backslash \sQ_D$ (see \cite[I.2]{gkz}) associated to the factorization $D=dd'$ defined by
\begin{equation}
	\chi_d(Q) =
	\begin{dcases}
		\pmfrac{d}{n} & \text{ if } (a,b,c,d)=1 \text{ and } Q \text{ represents }n \text{ and }(d,n)=1, \\
		\ 0 & \text{ if } (a,b,c,d)>1.
	\end{dcases}
\end{equation}
If $D$ is fundamental then $\chi_d=\chi_{d'}$ is the usual genus character, and there is exactly one such character for each such factorization.

As mentioned in the introduction, we need to regularize\footnote{The published version of this paper uses the regularization\[\int_{\sF_A} j_m(z) \mfrac{dxdy}{y^2} := \lim_{Y\to\infty} \int_{\sF_{A,Y}} j_m(z) \mfrac{dxdy}{y^2},\] where $\sF_{A,Y}$ is the surface $\sF_A$ truncated at height $Y$. In \cite[Lemma~2.5]{kalia-kumar}, it is shown that this natural regularization is related to \eqref{eq:int-jm-nu} by \[\int_\sF j_m(z) \nu_Q(z) \mfrac{dxdy}{y^2}=-\lim_{Y\to\infty} \int_{\sF_{A,Y}} j_m(z) \mfrac{dxdy}{y^2}-8\pi {\mathrm m}_Q \sigma_1(m),\] where $\mathrm m_Q = \mathrm m_A = n_1 + \ldots + n_\ell$ (see the discussion following \eqref{eq:gamma-w-ts}).} the surface integrals of $j_m$.
Following \cite[(70)]{dit-kronecker}, for $z,\tau\in \H$ we define
\begin{equation}
	K(z,\tau) := \frac{j'(\tau)}{j(z)-j(\tau)},
\end{equation}
where $j':=\frac 1{2\pi i}\frac{dj}{dz}$.
This function transforms on $\Gamma_1$ with weight $0$ in $z$ and weight $2$ in $\tau$.
For each indefinite quadratic form $Q$ define
\begin{equation}
	\nu_Q(z) := \int_{\mathcal C_Q} K(z,\tau) d\tau.
\end{equation}
As explained in \cite{dit-kronecker}, 
for $z\notin C_Q$ the value of $\nu_Q(z)$ is an integer which counts with signs the number of crossings that a path from $i\infty$ to $z$ in $\sF$ makes with $\mathcal C_Q$.
Furthermore, $\nu_Q(z)$ is $\Gamma_1$-invariant and is identically zero for $\im z$ sufficiently large.
It follows that the integral
\begin{equation} \label{eq:int-jm-nu}
	\int_\sF j_m(z) \nu_Q(z) \mfrac{dxdy}{y^2}
\end{equation}
converges, providing the desired regularization.

The following theorem generalizes Theorems~\ref{thm:main-thm-cycle-int} and \ref{thm:main-thm-surf} to more general discriminants.

\begin{theorem} \label{thm:main-gen-disc}
For each positive nonsquare discriminant $D$, let $D=dd'$ be any factorization into discriminants such that $d$ is fundamental.
Let $m$ be any positive integer.
If $d$ is positive, we have
\begin{equation} \label{eq:gen-disc-cycle-int}
	\sum_{Q\in \Gamma_1\backslash \sQ_D} \chi_d(Q) \int_{\mathcal C_Q} j_m(z) \mfrac{|dz|}{y} = -24 \, \delta_{d} \, \sigma_1(m) \cdot 2h(D) \log \ep_D + O\left( m^{\frac {8}{9}}D^{\frac{13}{27}}(mD)^\ep \right),
\end{equation}
while if $d$ is negative, we have
\begin{equation} \label{eq:gen-disc-surf}
	\mfrac{1}{4\pi}\sum_{Q\in \Gamma_1\backslash \sQ_D} \chi_d(Q) \int_{\sF} j_m(z) \nu_Q(z) \mfrac{dxdy}{y^2} = -24 \sigma_1(m) \frac{h(d)h(d')}{\omega_d \omega_{d'}} + O\left( m^{\frac {8}{9}}D^{\frac{13}{27}}(mD)^\ep \right).
\end{equation}
\end{theorem}

To deduce Theorem~\ref{thm:main-gen-disc} from Theorem~\ref{thm:main-kloo} we require several results from \cite[\S 8--9]{DIT-geom}, which we borrow from freely here.
For $m\geq 0$, let $F_{-m}(z,s)$ denote the index $-m$ nonholomorphic Poincar\'e series 
and let
\begin{equation}
	j_m(z,s) := 2\pi m^{\frac 12} F_{-m}(z,s) - \frac{2\pi m^{1-s} \sigma_{2s-1}(m)}{\pi^{-(s+\frac 12)}\Gamma(s+\frac 12)\zeta(2s-1)} F_0(z,s).
\end{equation}
For $m\geq 1$ the Fourier expansion of $F_{-m}(z,s)$ shows that it has an analytic continuation to $\re(s)>\frac 34$. 
In particular, $F_{-m}(z,1)$ is holomorphic as a function of $z$.
Furthermore, $F_0(z,s)$ is the nonholomorphic Eisenstein series of weight $\frac 12$, and we have
\begin{equation}
	\lim_{s\to 1} \frac{2\pi m^{1-s} \sigma_{2s-1}(m)}{\pi^{-(s+\frac 12)}\Gamma(s+\frac 12)\zeta(2s-1)} F_0(z,s) = 24\sigma_1(m).
\end{equation}
A computation then shows that $j_m(z)=j_m(z,1)$ for $m\geq 1$ (see \cite[(4.11)]{DIT-cycle}).

Since the length of $\mathcal C_Q$ is $2\log \ep_D$ for every $Q\in \sQ_D$, we have
\begin{equation}
	\sum_{Q\in \Gamma_1\backslash \sQ_D} \chi_d(Q) \int_{\mathcal C_Q}  \mfrac{|dz|}{y} = 2\log\ep_D \sum_{Q\in \Gamma_1\backslash \sQ_D} \chi_d(Q) = 2\delta_d h(D)\log\ep_D.
\end{equation}
By Corollary~4 of \cite{dit-kronecker}, we have (note that $\mathscr w_d=2\omega_d$ in that paper)
\begin{equation}
	\mfrac{1}{4\pi} \sum_{Q\in \Gamma_1\backslash \sQ_D} \chi_d(Q) \int_{\sF} \nu_Q(z) \mfrac{dxdy}{y^2} = \frac{h(d)h(d')}{\omega_d \omega_d'}.
\end{equation}
So to prove Theorem~\ref{thm:main-gen-disc} it suffices to show that
\begin{equation} \label{eq:cyc-int-F}
	\sqrt m \sum_{Q\in \Gamma_1\backslash \sQ_D} \chi_d(Q) \int_{\mathcal C_Q} F_{-m}(z,1) \mfrac{|dz|}{y} \ll m^{\frac {8}{9}}D^{\frac{13}{27}}(mD)^\ep
\end{equation}
and
\begin{equation} \label{eq:surf-int-F}
	\sqrt m \sum_{Q\in \Gamma_1\backslash \sQ_D} \chi_d(Q) \int_{\sF} F_{-m}(z,1) \nu_Q(z) \mfrac{dxdy}{y^2} \ll m^{\frac {8}{9}}D^{\frac{13}{27}}(mD)^\ep.
\end{equation}

We will prove \eqref{eq:cyc-int-F}--\eqref{eq:surf-int-F} by relating the integrals of $F_{-m}(z,1)$ to the quadratic Weyl sums
\begin{equation}
	T_m(d',d;c) := \sum_{\substack{b\bmod c \\ b^2\equiv D\bmod c}} \chi_d\left(\left[\tfrac c4,b,\tfrac{b^2-D}{c}\right]\right) e\pmfrac{2mb}{c}.
\end{equation}
Here we are still assuming that $D=dd'$ with $d$ fundamental. 
Note that $T_m(d',d;c)=S_m(d',d;c)$ in the notation of \cite{DIT-geom}; we have changed the notation here to avoid confusion with the Kloosterman sums.
The Weyl sums are related to the plus space Kloosterman sums via Kohnen's identity
\begin{equation} \label{eq:kohnen-identity}
	T_m(d,d';c) = \sum_{n\mid (m,\frac c4)} \pmfrac dn \sqrt{\mfrac {2n}{c}} \, S_{\frac 12}^+\left(d',\mfrac{m^2}{n^2}d;\mfrac cn\right)
\end{equation}
(see Lemma~8 of \cite{DIT-geom}).
The Weil bound \eqref{eq:weil-bound} for Kloosterman sums shows that
\begin{equation}
	T_m(d,d';c) \ll \gcd(d',m^2d,c)^{\frac 12} c^{\varepsilon}.
\end{equation}
A direct corollary of Theorem~\ref{thm:main-kloo} is the following bound for the Weyl sums.

\begin{theorem} \label{thm:sum-T-sums}
Suppose that $D=dd'$ is a positive nonsquare discriminant and that $d$ is a fundamental discriminant. 
Then for any $m\geq 1$ we have
\begin{equation} \label{eq:sum-T-sums}
\sum_{4\mid c\leq x} \frac{T_m(d,d';c)}{\sqrt c} \ll \left( x^{\frac 16} + D^{\frac 29}m^{\frac 13} \right) (mDx)^\varepsilon.
\end{equation}
\end{theorem}
\begin{proof}
When $d,d'$ are positive this is immediate from \eqref{eq:kohnen-identity} and the $k=\frac 12$ case of Theorem~\ref{thm:main-kloo}.
When $d,d'$ are negative we apply \eqref{eq:kloo-neg-k} after \eqref{eq:kohnen-identity}.
Then the estimate \eqref{eq:sum-T-sums} follows from the $k=-\frac 12$ case of Theorem~\ref{thm:main-kloo}.
\end{proof}

We are now ready to prove \eqref{eq:cyc-int-F}--\eqref{eq:surf-int-F}.

\begin{proof}[Proof of \eqref{eq:cyc-int-F}]
Let $J_\nu(x)$ denote the $J$-Bessel function
\begin{equation} \label{eq:J-def}
	J_{\nu}(2x) = \sum_{k=0}^{\infty} (-1)^k \frac{x^{2k+\nu}}{k! \Gamma(\nu+k+1)}.
\end{equation}
By Lemma~4 of \cite{DIT-geom} we have
\begin{equation} \label{eq:cyc-int-F-T-s}
	\sum_{Q\in \Gamma \backslash \sQ_D} \chi_d(Q)\int_{\mathcal C_Q} F_{-m}(z,s) \mfrac{|dz|}{y} = 
	2^{s-\frac 12} \frac{\Gamma(\frac s2)^2}{\Gamma(s)} D^{\frac 14} \sum_{0<c\equiv 0(4)} \frac{T_m(d',d;c)}{\sqrt c} J_{s-\frac 12}\pmfrac{4\pi m \sqrt D}{c}
\end{equation}
for $\re(s)>1$.
By \eqref{eq:J-def} we find that $J_{\nu}(1/x) \ll x^{-\nu}$ and $[J_{\nu}(1/x)]' \ll x^{-\nu-1}$ as $x\to \infty$ uniformly for $\nu\in [\frac 12, 1]$.
Let $a=4\pi m\sqrt{D}$ and let $N\geq a$.
Suppose that $s\in [1,\frac 32]$.
Then by partial summation and Theorem~\ref{thm:sum-T-sums} we have
\begin{multline}
	\sum_{4\mid c\geq N} \frac{T_m(d',d;c)}{\sqrt c} J_{s-\frac 12}\pmfrac{4\pi m \sqrt D}{c} 
	\\ = \lim_{x\to \infty} S(x) J_{s-\frac 12}\pmfrac{a}{x} - S(N) J_{s-\frac 12}\pmfrac aN - \int_{N}^\infty S(t) \left(J_{s-\frac 12}\pmfrac at\right)' \, dt \ll_a N^{-\frac 13+\ep},
\end{multline}
where $S(x)$ denotes the partial sum on the left-hand side of \eqref{eq:sum-T-sums}.
It follows that the sum on the right-hand side of \eqref{eq:cyc-int-F-T-s} converges uniformly for $s\in [1,\frac 32]$.
Since $J_\frac 12(x) = \sqrt{2/\pi x} \, \sin x$ we conclude that
\begin{equation} \label{eq:cyc-int-F-T}
	\sqrt m \sum_{Q\in \Gamma \backslash \sQ_D} \chi_d(Q)\int_{\mathcal C_Q} F_{-m}(z,1) \mfrac{|dz|}{y} = 
	\sum_{0<c\equiv 0(4)} T_m(d',d;c) \sin \pmfrac{4\pi m \sqrt D}{c}.
\end{equation}
We split the sum at $c=A$ with $A\ll m\sqrt D$.
Estimating the initial segment $c\leq A$ trivially, we obtain
\begin{equation} \label{eq:sum-T-triv}
	\sum_{c\leq A} T_m(d',d;c) \sin \pmfrac{4\pi m \sqrt D}{c} \ll A(mDA)^\ep.
\end{equation}
Then by partial summation we have
\begin{equation}
	\sum_{c > A} T_m(d',d;c) \sin \pmfrac{4\pi m \sqrt D}{c} = -S(A)\sqrt A\,\sin \pmfrac{4\pi m \sqrt D}{A} - \int_A^\infty S(t) \left( \sqrt t \, \sin \pmfrac{4\pi m\sqrt D}{t} \right)'  dt,
\end{equation}
where $S(x)$ denotes the partial sum on the left-hand side of \eqref{eq:sum-T-sums}.
Since
\begin{equation}
	\left( \sqrt t \, \sin \pmfrac{4\pi m\sqrt D}{t} \right)' \ll \mfrac{m\sqrt D}{t^{\frac 32}},
\end{equation}
we conclude that
\begin{equation} \label{eq:sum-T-tail}
	\sum_{c > A} T_m(d',d;c) \sin \pmfrac{4\pi m \sqrt D}{c}
	\ll \left( mD^{\frac 12}A^{-\frac 13} + m^{\frac 43} D^{\frac{13}{18}} A^{-\frac 12} \right) (mDA)^\varepsilon.
\end{equation}
Letting $A=m^aD^b$, we choose $a=\frac 89$ and $b=\frac{13}{27}$ to balance the exponents in \eqref{eq:sum-T-triv} and \eqref{eq:sum-T-tail}.
This, together with \eqref{eq:cyc-int-F-T}, yields \eqref{eq:cyc-int-F}.
\end{proof}

\begin{proof}[Proof of \eqref{eq:surf-int-F}]
Define $\sF_Y:=\sF \cap \{z:\im z\leq Y\}$.
Let $Q\in \sQ_D$, and let $Y$ be sufficiently large so that $\nu_Q(z)=0$ for $\im z > Y$ and so that the image of $\mathcal C_Q$ in $\sF$ is contained in $\sF_Y$.
Then for $\re s>1$ we have
\begin{equation}
	\int_\sF F_{-m}(z,s) \nu_Q(z) \mfrac{dxdy}{y^2} 
	= \int_{\mathcal C_Q} \int_{\sF_Y} F_{-m}(z,s) K(\tau,z) \mfrac{dxdy}{y^2} \, d\tau.
\end{equation}
The function $F_{-m}(z,s)$ satisfies the relation
\begin{equation}
	\Delta_0 F_{-m}(z,s) := -4y^2 \partial_{\bar z} \partial_z F_{-m}(z,s) = s(1-s) F_{-m}(z,s)
\end{equation}
(see \S 8 of \cite{DIT-geom}).
So by the proof of Lemma~1 of \cite{dit-kronecker} (essentially an application of Stokes' theorem), we find that\footnote{The published version of this paper contains an error: the functions $g$ and $h$ are omitted from the right-hand side of \eqref{eq:fix1}. The details of this corrected computation, including an explicit description of the functions $g$ and $h$, can be found in the proof of Proposition~2.4 of \cite{kalia-kumar}.}
\begin{equation} \label{eq:fix1}
	\mfrac{s(1-s)}{2} \int_{\sF_Y} F_{-m}(z,s) K(\tau,z) \mfrac{dxdy}{y^2} = i \partial_\tau F_{-m}(\tau,s) + g(\tau,s,Y) + h(\tau),
\end{equation}
for some functions $g$ and $h$ with $g(\tau,s,Y) \to 0$ as $Y\to \infty$.
It follows that
\begin{equation}
	\mfrac{s(1-s)}{2} \int_{\sF} F_{-m}(z,s)\nu_Q(z) \mfrac{dxdy}{y^2} = \int_{\mathcal C_Q} i \partial_z F_{-m}(z,s) dz + h(\tau).
\end{equation}
Differentiating with respect to $s$ and setting $s=1$ we conclude that
\begin{equation}
	\int_{\sF} F_{-m}(z,1) \nu_Q(z) \mfrac{dxdy}{y^2} = -2 \partial_s \int_{\mathcal C_Q} i\partial_z \, F_{-m}(z,s) \, dz \Big|_{s=1}.
\end{equation}

By Lemma~5 of \cite{DIT-geom} we have
\begin{equation} \label{eq:surf-int-F-T-s}
	\sum_{Q\in \Gamma \backslash \sQ_D} \chi_d(Q)\int_{\mathcal C_Q} i\partial_z F_{-m}(z,s) dz = 
	2^{s-\frac 12} \frac{\Gamma(\frac {s+1}2)^2}{\Gamma(s)} D^{\frac 14} \sum_{0<c\equiv 0(4)} \frac{T_m(d',d;c)}{\sqrt c} J_{s-\frac 12}\pmfrac{4\pi m \sqrt D}{c}.
\end{equation}
A straightforward computation involving \eqref{eq:J-def} shows that, uniformly for $s\in [1,\frac 32]$, we have
\begin{equation}
	\partial_s \left[ 2^{s-\frac 12} \frac{\Gamma(\frac {s+1}2)^2}{\Gamma(s)} J_{s-\frac 12}(x) \right]  \ll x^{s-\frac 12} \, |\!\log x| \quad \text{ as } \quad x\to0^+.
\end{equation}
Thus, an argument involving partial summation, as in the proof of \eqref{eq:cyc-int-F}, shows that we are justified in setting $s=1$, and we obtain
\begin{equation}
	\sqrt m \sum_{Q\in \Gamma \backslash \sQ_D} \chi_d(Q)\int_{\sF} F_{-m}(z,1) \nu_Q(z) \mfrac{dxdy}{y^2} = -\mfrac{2}{\pi} \sum_{0<c\equiv 0(4)} T_m(d',d;c) f \pmfrac{4\pi m \sqrt D}{c},
\end{equation}
where
\begin{equation}
	f(x) := \Ci(2x) \sin(x) - \Si(2x) \cos(x) + \log(2) \sin(x)
\end{equation}
and $\Ci$, $\Si$ are the cosine and sine integrals, respectively.
The remainder of the proof is quite similar to the proof of \eqref{eq:cyc-int-F} because we have
\begin{equation}
	f(x) \ll \min\{1,x|\!\log x|\} \quad \text{ and } \quad \left(\sqrt t \, f \pmfrac{4\pi m \sqrt D}{t} \right)' \ll \mfrac{m\sqrt D}{t^{\frac 32}} (mDt)^\ep.
\end{equation}
We omit the details.
\end{proof}

\section{Background} \label{sec:background}

In this section we recall several facts about automorphic functions which transform according to multiplier systems of half-integral weight $k$, and the spectrum of the hyperbolic Laplacian $\Delta_k$ in this setting.
For more details see \cite{duke-friedlander-iwaniec,sarnak-additive,proskurin-new,aa-kloosterman} along with the original papers of Maass \cite{maass1,maass2}, Roelcke \cite{roelcke2}, and Selberg \cite{selberg-harmonic,selberg-estimation}.

Let $\Gamma = \Gamma_0(N)$ for some $N\geq 1$, and let $k$ be a real number.
We say that $\nu:\Gamma\to \C^\times$ is a multiplier system of weight $k$ if
\begin{enumerate}[\hspace{1.5em}(i)]\setlength\itemsep{.4em} 
	\item $|\nu|=1$,
	\item $\nu(-I)=e^{-\pi i k}$, and
	\item $\nu(\gamma_1 \gamma_2) = w(\gamma_1,\gamma_2) \nu(\gamma_1)\nu(\gamma_2) $ for all $\gamma_1,\gamma_2\in \Gamma$,
	where
	\begin{equation} \label{eq:w-def}
		w(\gamma_1,\gamma_2) = j(\gamma_2,z)^k j(\gamma_1,\gamma_2z)^k j(\gamma_1\gamma_2,z)^{-k},
	\end{equation}
	and $j(\gamma,z)$ is the automorphy factor
	\begin{equation}
		j(\gamma,z) := \mfrac{cz+d}{|cz+d|} = e^{i\arg(cz+d)}.
	\end{equation}
\end{enumerate}
If $\nu$ is a multiplier system of weight $k$, then
 $\bar\nu$ is a multiplier system of weight $-k$.

The group $\SL_2(\R)$ acts on $\H$ via $\pmatrix abcd z = \frac{az+b}{cz+d}$.
The cusps of $\Gamma$ are those points in the extended upper half-plane $\H^*$ which are fixed by parabolic elements of $\Gamma$.
Given a cusp $\a$ of $\Gamma$ let $\Gamma_\a:=\{\gamma\in\Gamma:\gamma\a=\a\}$ denote its stabilizer in $\Gamma$, and let $\sigma_\a$ denote any element of $\SL_2(\R)$ satisfying $\sigma_\a \infty = \a$ and $\sigma_\a^{-1}\Gamma_\a\sigma_\a=\Gamma_\infty$.
Define $\kappa_{\a} = \kappa_{\nu,\a}$ by the conditions
\[
	\nu\left(\sigma_\a\pmatrix 1101 \sigma_\a^{-1} \right) = e(-\kappa_{\a}) \quad \text{ and } \quad 0\leq \kappa_\a <1.
\]
We say that $\a$ is singular with respect to $\nu$ if $\nu$ is trivial on $\Gamma_\a$, that is, if $\kappa_{\nu,\a}=0$.
Note that if $\kappa_{\nu,\a}>0$ then
\begin{equation} \label{eq:kappa-nu-bar}
	\kappa_{\bar\nu,\a} = 1-\kappa_{\nu,\a}.
\end{equation}

We are primarily interested in the multiplier system $\nu_\theta$ of weight $\frac 12$ (and its conjugate $\bar\nu_\theta = \nu_\theta^{-1}$ of weight $-\frac 12$) on $\Gamma_0(4)$ defined by
\[
	\theta(\gamma z) = \nu_\theta(\gamma) \sqrt{cz+d} \, \theta(z),
\]
where
\[
	\theta(z) := \sum_{n\in\Z} e(n^2 z).
\]
Explicitly, we have
\begin{equation} \label{eq:def-theta-mult}
	\nu_\theta \ppMatrix **cd = \pfrac cd \varepsilon_d^{-1},
\end{equation}
where $\ptfrac\cdot\cdot$ is the extension of the Kronecker symbol given e.g. in \cite{shimura} and
\[
	\varepsilon_d = \pmfrac{-1}d^{\frac 12} =
	\begin{cases}
		1 & \text{ if }d\equiv 1\pmod 4, \\
		i & \text{ if }d\equiv 3\pmod 4.
	\end{cases}
\]

For $\gamma\in \SL_2(\R)$ we define the weight $k$ slash operator   by
\[
	f\sl_k \gamma := j(\gamma,z)^{-k} f(\gamma z).
\]
The weight $k$ hyperbolic Laplacian
\begin{equation*}
	\Delta_k := y^2 \bigg( \frac{\partial^2}{\partial x^2} + \frac{\partial^2}{\partial y^2} \bigg) - iky \frac{\partial}{\partial x}
\end{equation*}
 commutes with the weight $k$ slash operator for every $\gamma\in \SL_2(\R)$.
A real analytic function $f:\H\to\C$ is an eigenfunction of $\Delta_k$ with eigenvalue $\lambda$ if
\begin{equation} \label{eq:delta-k-f-lambda-f}
	\Delta_k f + \lambda f = 0.
\end{equation}
If $f$ satisfies \eqref{eq:delta-k-f-lambda-f} then for 
 notational convenience we write
\[
	\lambda = \mfrac 14 + r^2,
\]
and we refer to $r$ as the spectral parameter of $f$.

A function $f:\H\to\C$ is automorphic of weight $k$ and multiplier $\nu$ for $\Gamma$ if
\begin{equation} \label{eq:transformation-law}
	f \sl_k \gamma = \nu(\gamma) f \qquad \text{ for all }\gamma\in\Gamma.
\end{equation}
Let $\mathcal{A}_k(N,\nu)$ denote the space of all such functions.
A smooth automorphic function which is also an eigenfunction of $\Delta_k$ and which has at most polynomial growth at the cusps of $\Gamma$
is called a Maass form.
We let $\mathcal{A}_k(N,\nu,r)$ denote the vector space of Maass forms with spectral parameter $r$.
Complex conjugation $f\to \bar f$ gives a bijection $\mathcal{A}_k(N,\nu,r) \longleftrightarrow \mathcal{A}_{-k}(N,\bar\nu,r)$.

If $f \in \mathcal{A}_k(n,\nu,r)$, then $f\sl_k \sigma_\a$ satisfies $(f\sl_k \sigma_\a)(z+1)=e(-\kappa_\a)(f\sl_k \sigma_\a)(z)$.
For $n\in\Z$ define
\begin{equation}\label{eq:n_a-def}
	n_\a := n-\kappa_\a.
\end{equation}
Then $f$ has a Fourier expansion at the cusp $\a$ of the form
\begin{equation} \label{eq:fourier-exp-0}
	(f\sl_k \sigma_\a)(z) = \rho_{f,\a}(0)y^{\frac 12+ir} + \rho_{f,\a}'(0) y^{\frac 12-ir} + \sum_{n_\a\neq 0} \rho_{f,\a}(n) W_{\frac k2\sgn(n),ir}(4\pi|n_\a|y) e(n_\a x),
	\end{equation}
where $W_{\kappa,\mu}(y)$ is the $W$-Whittaker function.
When the weight is $0$, many authors normalize the Fourier coefficients so that $\rho_{f,\a}(n)$ is the coefficient of 
$\sqrt y \, K_{ir}(2\pi|n_\a|y)$, where $K_{\nu}(y)$ is the $K$-Bessel function.
Using the relation
\begin{equation}
	W_{0,\mu}(y) = \frac {\sqrt y}{\sqrt \pi} \, K_{\mu}(y/2),
\end{equation}
we see that this has the effect of multiplying $\rho_{f,\a}(n)$ by $2|n_\a|^{1/2}$.

Let $\mathcal{L}_k(\nu)$ 
denote the $L^2$-space of automorphic functions with respect to the Petersson inner product
\begin{equation}\label{eq:inner_prod_def}
	\langle f,g \rangle := \int_{\Gamma\backslash\H} f(z) \bar{g(z)} \, d\mu, \qquad d\mu := \mfrac{dx\,dy}{y^2},
\end{equation}
and let $\mathcal L_k(\nu,\lambda)$ denote the $\lambda$-eigenspace.
The spectrum of $\Delta_k$ is real and contained in $[\lambda_0(k),\infty)$, where $\lambda_0(k) := \frac {|k|}2 (1-\frac{|k|}2)$.
The minimal eigenvalue $\lambda_0(k)$ occurs if and only if there is a holomorphic modular form $F$ of weight $|k|$ and multiplier $\nu$, in which case
\[
	f_0(z) =
	\begin{cases}
		y^{\frac k2} \, F(z) & \text{ if }k\geq 0, \\
		y^{-\frac k2} \, \bar F(z) & \text{ if }k<0,
	\end{cases}
\]
is the corresponding eigenfunction.
When $k=\pm \frac 12$ and $\nu=\nu_\theta^{2k}$, the eigenspace $\mathcal L_k(\nu,\frac{3}{16})$ is one-dimensional, spanned by $y^{\frac 14}\theta(z)$ if $k=\frac 12$ and $y^{-\frac 14}\bar\theta(z)$ if $k=-\frac 12$.

The spectrum of $\Delta_k$ on $\mathcal{L}_k(\nu)$ consists of an absolutely continuous spectrum of multiplicity equal to the number of singular cusps, and a discrete spectrum of finite multiplicity.
The Eisenstein series, of which there is one for each singular cusp $\a$, give rise to the continuous spectrum, which  is bounded below by $1/4$.
Let $\a$ be a singular cusp.
The Eisenstein series for the cusp $\a$ is defined by
\begin{equation}
	E_\a(z,s) := \sum_{\gamma \in \Gamma_\a\backslash \Gamma_\infty} \bar\nu(\gamma) \bar w(\sigma_\a^{-1},\gamma) j(\sigma_\a^{-1}\gamma,z)^{-k} \im(\sigma_\a^{-1}\gamma z)^s.
\end{equation}
If $\b$ is any cusp, the Fourier expansion for $E_\a$ at the cusp $\b$ is given by
\begin{equation}
	j(\sigma_\b,z)^{-k} E_\a(z,s) = \delta_{\a=\b}y^s + \delta_{\kappa_\b=0}\phi_{\a\b}(0,s)y^{1-s} + \sum_{n_\b\neq 0} \phi_{\a\b}(n,s) W_{\frac k2 \sgn(n),s-\frac 12}(4\pi |n_\b|y) e(n_\b x),
\end{equation}
where
\begin{equation} \label{eq:phi-ab-def}
	\phi_{\a\b}(n,s) = 
	\begin{dcases}
		\frac{e(-\frac k4)\pi^s |n|^{s-1} }{\Gamma(s+\frac k2\sgn(n))}  \sum_{c\in \mathcal C(\a,\b)} \frac{S_{\a\b}(0,n,c,\nu)}{c^{2s}}  & \text{ if }n_\b \neq 0, \\
		\frac{e(-\frac k4)\pi 4^{1-s} \Gamma(2s-1)}{\Gamma(s+\frac k2)\Gamma(s-\frac k2)}\sum_{c\in \mathcal C(\a,\b)} \frac{S_{\a\b}(0,0,c,\nu)}{c^{2s}} & \text{ if }n_\b =0.
	\end{dcases}
\end{equation}
Here $\mathcal C(\a,\b)=\{c>0:\pmatrix **c* \in \sigma_\a^{-1}\Gamma \sigma_\b\}$ is the set of allowed moduli and $S_{\a\b}(m,n,c,\nu)$ is the Kloosterman sum (defined for any cusp pair $\a\b$)
\begin{equation} \label{eq:S-ab-ref}
	S_{\a\b}(m,n,c,\nu) := \sum_{\gamma = \pmatrix abcd \in \Gamma_\infty \backslash \sigma_\a^{-1}\Gamma\sigma_\b / \Gamma_\infty} \bar\nu_{\a\b}(\gamma) e\pfrac{m_\a a + n_\b d}{c},
\end{equation}
where
\begin{equation}
	\nu_{\a\b}(\gamma) = \nu(\sigma_\a \gamma \sigma_\b^{-1}) \frac{w(\sigma_\a \gamma \sigma_\b^{-1}, \sigma_\b)}{w(\sigma_\a,\gamma)}.
\end{equation}
The coefficients $\phi_{\a\b}(n,s)$ can be meromorphically continued to the entire $s$-plane and, in particular, are well-defined on the line $\re(s)=\frac 12$.
In Section~\ref{sec:plus} we will evaluate certain linear combinations of the coefficients $\phi_{\a\b}(n,\tfrac 12\pm ir)$ in terms of Dirichlet $L$-functions in the cases $k=\pm \frac12$ and $\nu=\nu_\theta^{2k}$.

Let $\mathcal{V}_k(\nu)$ denote the orthogonal complement in $\mathcal{L}_k(\nu)$ of the space generated by Eisenstein series.
The spectrum of $\Delta_k$ on $\mathcal{V}_k(\nu)$ is countable and of finite multiplicity.
The exceptional eigenvalues are those which lie in  $(\lambda_0(k),1/4)$ (conjecturally, the set of exceptional eigenvalues is empty).
The subspace $\mathcal{V}_k(\nu)$ consists of functions $f$ which decay exponentially at every cusp; equivalently, the zeroth Fourier coefficient of $f$ at each singular cusp vanishes.
Eigenfunctions of $\Delta_k$ in $\mathcal{V}_k(\nu)$ are called Maass cusp forms.

Let $\{f_j\}$ be an orthonormal 
basis of $\mathcal{V}_k(\nu)$, and for each $j$ let $\lambda_j = \frac 14+r_j^2$
denote the  Laplace eigenvalue and $\{\rho_{j,\a}(n)\}$ the Fourier coefficients.
Weyl's law describes the distribution of the spectral parameters $r_j$.
Theorem~2.28 of \cite{hejhal-stf2} shows that
\begin{equation}
	\sum_{0\leq r_j\leq T} 1 - \mfrac{1}{4\pi} \int_{-T}^T \mfrac{\varphi'}{\varphi}\Big(\mfrac 12+it\Big) \, dt 
	= \mfrac{\operatorname{vol}(\Gamma\backslash\H)}{4\pi} \, T^2 - \mfrac{K_0}{\pi} \, T\log T + O(T),
\end{equation}
where $\varphi(s)$ and $K_0$ are the determinant (see \cite[p. 298]{hejhal-stf2}) and dimension (see \cite[p. 281]{hejhal-stf2}), respectively, of the scattering matrix $\Phi(s)$ whose entries are given in terms of constant terms of Eisenstein series.

\section{An estimate for coefficients of Maass cusp forms}\label{sec:MVE}

In this section we prove a general theorem which applies to the Fourier coefficients at the cusp $\a$ of weight $\pm \frac12$ Maass cusp forms with multiplier $\nu$ for $\Gamma=\Gamma_0(N)$.
We assume that the bound
	\begin{equation} \label{eq:mve-weil-assumption}
		\sum_{c>0} \frac{|S_{\a\a}(n,n,c,\nu)|}{c^{1+\beta}} \ll_{\nu} n^{\varepsilon}
	\end{equation}
holds
for some $\beta=\beta_{\nu,\a} \in (1/2,1)$.
A similar estimate was proved in \cite[Theorem~3.1]{aa-kloosterman}, but the following theorem improves the bound in the $x$-aspect when $k=\frac12$.
The proof given here is also considerably shorter.

\begin{theorem} \label{thm:mve-general}
Suppose that $k=\pm \frac12$ and that $\nu$ is a multiplier system of weight $k$ which satisfies \eqref{eq:mve-weil-assumption}. 
Fix an orthonormal basis of cusp forms $\{u_j\}$ for $\mathcal{V}_k(\nu)$.
For each $j$, let $\rho_{j,\a}(n)$ denote the $n$-th Fourier coefficient of $u_j$ at $\a$ and let $r_j$ denote the spectral parameter.
Then for all $n\geq 1$ we have
\begin{equation} \label{eq:mve}
	n_\a \sum_{x\leq r_j\leq 2x} |\rho_{j,\a}(n)|^2 e^{-\pi r_j} \ll
		x^{-k}\left(x^2 + n^{\beta+\varepsilon} x^{1-2\beta} \log^\beta x\right).
\end{equation}
\end{theorem}

We begin with an auxiliary version of Kuznetsov's formula  (\cite[\S 5]{kuznetsov}) which is Lemma 3 of \cite{proskurin-new} with $m=n$, $t\mapsto 2t$, and $\sigma=1$ (see \cite[Section~3]{aa-kloosterman} for justification of the latter).
While Proskurin assumes that $k>0$ throughout his paper, this lemma is still valid for $k<0$ by the same proof, and straightforward modifications give the result for an arbitrary cusp $\a$.

\begin{lemma}\label{lem:pre_kusnetsov}
With the assumptions of Theorem~\ref{thm:mve-general}, and for any $t\in\R$ we have
\begin{multline} \label{eq:proskurin-lemma-3}
	\frac{2\pi^2 n_\a}{|\Gamma(1-\frac k2+it)|^2} \left[ \sum_{r_j} \frac{|\rho_{j,\a}(n)|^2}{\cosh 2\pi r_j+\cosh 2\pi t}   \right. \\
	 + \left. \frac{1}{4} \sum_{\c} \int_{-\infty}^\infty \frac{\left|\phi_{\c\a}(n,\tfrac 12+ir)\right|^2}{(\cosh 2\pi r+\cosh 2\pi t)|\Gamma(\frac{k+1}{2}+ir)|^2} \, dr \right] \\
	= \frac{1}{4\pi} + \frac{2n_\a}{i^{k+1}}\sum_{c>0} \frac{S_{\a\a}(n,n,c,\nu)}{c^{2}} \int_L K_{2it} \left(\mfrac{4\pi n_\a}{c} \, q\right) q^{k-1} \, dq ,
\end{multline}
where $\sum_{\c}$ is a sum over singular cusps, and
 $L$ is the semicircular contour $|q|=1$ with $\re(q)>0$, from $-i$ to $i$.
\end{lemma}

To prove Theorem~\ref{thm:mve-general} we follow the method of Motohashi \cite[Section~2]{motohashi-fourth}.
We begin by evaluating the integral on the right-hand side of \eqref{eq:proskurin-lemma-3} via the following lemma.
For the remainder of this section we frequently use the notation $\int_{(\xi)}$ to denote $\int_{\xi-i\infty}^{\xi+i\infty}$.

\begin{lemma} \label{lem:int_L}
Let $k=\pm \frac 12$.
Suppose that $a>0$, $\xi>\frac k2$. Then
\begin{equation}
	2\int_L K_{2it}(2aq) q^{k-1} \, dq = \frac 1{2\pi} \int_{(\xi)} \frac{\sin(\pi s-\frac{\pi k}{2})}{s-\frac k2} \Gamma(s+it) \Gamma(s-it) a^{-2s} \, ds.
\end{equation}
\end{lemma}

\begin{proof}
For any $\xi>0$ we have the Mellin-Barnes integral representation \cite[(10.32.13)]{nist}
\begin{equation}
	2 K_{2it}(2z) = \frac{1}{2\pi i} \int_{(\xi)} \Gamma(s) \Gamma(s-2it) z^{2it-2s} \, ds,
\end{equation}
which is valid for $|\arg z|<\frac\pi 2$.
It follows that
\begin{align*}
	2\int_L K_{2it}(2aq) q^{k-1} \, dq 
	&= \frac{1}{2\pi i} \int_{(\xi)} \Gamma(s) \Gamma(s-2it) a^{2it-2s} \int_{L} q^{2it-2s+k-1} \, dq\, ds \\
	&= \frac{1}{2\pi} \int_{(\xi)} \Gamma(s) \Gamma(s-2it) a^{2it-2s} \frac{\sin(\pi(s-it-\frac{k}{2}))}{s-it-\frac{k}{2}} \, ds.
\end{align*}
The lemma follows after replacing $s$ by $s+it$.
\end{proof}

Let $K$ be a large positive real number.
In \eqref{eq:proskurin-lemma-3} 
 we multiply by the positive weight
\[
	e^{-(t/K)^2} - e^{-(2t/K)^2}
\]
and integrate on $t$ over $\R$.
Applying Lemma \ref{lem:int_L} to the result (and noting that all terms on the left-hand side are positive), we obtain
\begin{equation} \label{eq:mve-1}
	n_\a \sum_{r_j} |a_j(n)|^2 h_K(r_j) 
	  \ll K
	+ \sum_{c>0} \frac{|S(n,n,c,\nu)|}{c} \left|M_k\left(K,\mfrac{2\pi n_\a}{c}\right)\right|,
\end{equation}
where
\begin{equation}\label{eq:hxrdef}
	h_K(r) := \int_{-\infty}^{\infty} \frac{ e^{-(t/K)^2} - e^{-(2t/K)^2} }{ |\Gamma(1-\frac k2+it)|^2 (\ch 2\pi r+\ch 2\pi t) } \, dt
\end{equation}
and
\begin{equation}\label{eq:ixadef}
M(K,a) = \int_{-\infty}^\infty \left(e^{-(t/K)^2} - e^{-(2t/K)^2}\right) \int_{(\xi)} \frac{\sin(\pi s-\frac{\pi k}{2})}{s-\frac k2} \Gamma(s+it) \Gamma(s-it) a^{1-2s} \, ds \, dt.
\end{equation}

We will make use of the following well-known estimate for oscillatory integrals (see, for instance, \cite[Chapter IV]{titchmarsh}).

\begin{lemma} \label{lem:first-deriv}
Suppose that $F$ and $G$ are real-valued functions on $[a,b]$ with $F$ differentiable, such that $G(x)/F'(x)$ is monotonic. 
If $|F'(x)/G(x)|\geq m>0$ then
\[
	\int_a^b G(x) e(F(x)) \, dx \ll \frac{1}{m}.
\]
\end{lemma}

\begin{proposition} \label{prop:mve-int-est-pos}
Let $K$ be a large positive real number.
Suppose that $k=\pm \frac 12$ and let $M(K,a)$ be as above.
For $a>0$ we have
\begin{equation} \label{eq:M-k-a-est}
	M(K,a) \ll \min\left( 1, \frac{a \log K}{K^{2}} \right).
\end{equation}
\end{proposition}

\begin{proof}
Starting with the integral representation \cite[(5.12.1)]{nist}
\begin{equation}
	\Gamma(s+it)\Gamma(s-it) = \Gamma(2s)\int_0^1 y^{s+it-1}(1-y)^{s-it-1} \, dy,
\end{equation}
we interchange the order of integration, putting the integral on $t$ inside, and find that the integral on $t$ equals
\begin{align}
	T(K,y) &= \int_{-\infty}^{\infty} \pfrac{y}{1-y}^{it} \left(e^{-(t/K)^2} - e^{-(2t/K)^2}\right) \, dt \\
		&= K e^{-\frac 14 K^2 \log^2 \pfrac{y}{1-y}} - \mfrac 12 K e^{-\frac 1{16} K^2 \log^2 \pfrac{y}{1-y} }.
\end{align}
Hence
\begin{equation} \label{eq:Mka-1}
	M(K,a) = \int_0^1 \frac{T(K,y)}{y(1-y)} \int_{(\xi)} \frac{\sin(\pi s-\frac{\pi k}{2})}{s-\frac k2} \Gamma(2s) [y(1-y)]^s a^{1-2s} \, ds \, dy.
\end{equation}
To evaluate the inner integral, we use that
\begin{equation}
	\frac{u^{k-2s}}{s-\frac k2} = 2\int_u^\infty t^{-2s+k-1} \, dt.
\end{equation}
Setting $u=a \,[y(1-y)]^{-\frac12}$, the integral on $s$ in \eqref{eq:Mka-1} equals
\begin{equation}
	2a u^{-k} \int_u^\infty t^{k-1} \int_{(\xi)} \sin(\pi s-\tfrac{\pi k}{2}) \Gamma(2s) t^{-2s} \, ds \, dt = a f_k(u),
\end{equation}
where
\begin{equation}
	f_k(u) = \cos\pmfrac{\pi k}{2} u^{-k} \int_u^\infty t^{k-1} \sin t \, dt -\sin\pmfrac{\pi k}2 u^{-k} \int_u^\infty t^{k-1}\cos t \, dt.
\end{equation}
Finally, we set $z=K\log\frac{y}{1-y}$ to obtain
\begin{equation}
	M(K,a) = a \int_{-\infty}^{\infty} \left( e^{-z^2/4} - \mfrac 12 e^{-z^2/16} \right) f_k\left(2a\cosh\pmfrac z{2K}\right) \, dz.
\end{equation}

We claim that $f_k(u) \ll \min(1,1/u)$.
For $u \geq 1$ this follows from Lemma~\ref{lem:first-deriv}.
Suppose that $u \leq 1$.
In the case $k=-\frac 12$, we have $f_k(u)\ll 1$ by estimating the integrals trivially.
When $k=\frac 12$ a computation shows that
\begin{equation}
	f_{\frac 12}(u) =  \frac{ \sqrt\pi \, C\big(\sqrt{2u/\pi} \big) - \sqrt\pi \, S\big(\sqrt{2u/\pi}\big)}{\sqrt u},
\end{equation}
where $C(x)$ and $S(x)$ are the Fresnel integrals \cite[\S 7.2]{nist}.
It follows that $f_{\frac 12}(u)\ll 1$.

From the estimate $f_k(u) \ll \min(1,1/u)$ it follows that 
\begin{equation}
	M(K,a) \ll 1.
\end{equation}
Now suppose that $a\ll K^2$.
In this case we add and subtract $f_k(2a)$ from the integrand and notice that 
\[
	\int_{-\infty}^{\infty} \left( e^{-z^2/4} - \mfrac 12 e^{-z^2/16} \right) \, dz = 0,
\]
so
\begin{equation}
	M(K,a) \ll a \int_0^\infty e^{-z^2/16} \left| f_k(2a) - f_k\left(2a \cosh\pmfrac z{2K}\right) \right| \, dz.
\end{equation}
Let $T=c\sqrt{\log K}$ with $c$ a large constant, and let $F(z) = f_k(2a) - f_k(2a \cosh z)$.
Then $F(0)=F'(0)=0$, so for $|z|\leq T/K$ we have
\begin{equation} \label{eq:F-z-bound}
	F(z) \ll z^2 \max_{|w|\leq T/K}|F''(w)|.
\end{equation}
Since
\begin{equation}
	F''(w) \ll a \cosh w |f_k'(2a \cosh w)| + a^2 \sinh^2 w |f_k''(2a \cosh w)|
\end{equation}
and, by Lemma~\ref{lem:first-deriv},
\begin{equation}
	f_k'(u), f_k''(u) \ll u^{-1},
\end{equation}
we conclude that 
\begin{equation} \label{eq:F''-w-bound}
	F''(w) \ll a \sinh(T/K)\tanh(T/K) \ll \frac{a T^2}{K^2} \ll T^2.
\end{equation}
By \eqref{eq:F-z-bound} and \eqref{eq:F''-w-bound} we have
\begin{equation}
	a \int_0^T e^{-z^2/16} \left|F\pmfrac{z}{2K}\right| \, dz \ll \frac{a T^2}{K^2} \int_0^\infty z^2 e^{-z^2/16} \, dz  \ll  \frac{a T^2}{K^2}
\end{equation}
and by $f_k(u)\ll 1$ we have
\begin{equation}
	a \int_T^\infty e^{-z^2/16} \left| f_k(2a) - f_k\left(2a \cosh \pmfrac z{2K}\right) \right| \, dz \ll a \int_T^\infty e^{-z^2/16} \, dz \ll a \, e^{-T^2/16}.
\end{equation}
With our choice of $T$ this yields \eqref{eq:M-k-a-est}.
\end{proof}

\begin{proof}[Proof of Theorem~\ref{thm:mve-general}]
First note that when $r\sim x$ we have $h_x(r) \gg e^{-\pi r} x^{k-1}$, where $h_x(r)$ is defined in \eqref{eq:hxrdef}, so
by \eqref{eq:mve-1} and positivity we have
\begin{equation}
n_\a x^{k} \sum_{x\leq r_j\leq 2x} |\rho_{j,\a}(n)| e^{-\pi r_j} \ll x^2 + x \sum_{c>0} \frac{|S_{\a\a}(n,n,c,\nu)|}{c} \left|M_k\left(x,\mfrac {2\pi n_\a}c \right)\right|.
\end{equation}
Let $\beta$ be as in \eqref{eq:mve-weil-assumption}.
By Proposition~\ref{prop:mve-int-est-pos} we have
\begin{equation}
	M_k(x,a) \ll \min\left(1,\frac{a \log x}{x^2}\right) \ll \frac{a^\beta \log^\beta x}{x^{2\beta}},
\end{equation}
from which it follows that
\begin{align}
	x \sum_{c>0} \frac{|S_{\a\a}(n,n,c,\nu)|}{c} \left|M_k\left(x,\mfrac {2\pi n_\a}c \right)\right| &\ll n_\a^{\beta} x^{1-2\beta} \log^\beta x \sum_{c>0} \frac{|S_{\a\a}(n,n,c,\nu)|}{c^{1+\beta}} \\
	&\ll n_\a^{\beta+\varepsilon} x^{1-2\beta} \log^\beta x.
\end{align}
The theorem follows.
\end{proof}


\section{The Kuznetsov formula for Kohnen's plus space} \label{sec:plus}

In this section we define the plus spaces of holomorphic and Maass cusp forms, and we prove an analogue of Kuznetsov's formula relating the Kloosterman sums $S_k^+(m,n,c)$ to the Fourier coefficients of such forms.
For the remainder of the paper we specialize to the case $\Gamma=\Gamma_0(4)$ with $(k,\nu) = (\frac 12, \nu_{\theta})$ or $(-\frac 12, \bar\nu_{\theta})$.
We will often write $k=\lambda+\frac 12$, and
to simplify notation, we write $\mathcal V_k = \mathcal V_k(\nu)$ and $\mathcal S_\ell=\mathcal S_\ell(\nu)$, where $S_\ell(\nu)$ is the space of holomorphic cusp forms of weight $\ell$ and multiplier $\nu$.
We fix once and for all a set of inequivalent representatives for the cusps of $\Gamma$, namely $\infty$, $0$, and $\frac 12$,
with associated scaling matrices
\begin{equation}
	\sigma_\infty = \pMatrix 1001, \qquad
	\sigma_0 = \pMatrix0{-\frac12}20, \qquad
	\sigma_{\frac 12} = \pMatrix 1{-\frac12}20.
\end{equation}
Then
\begin{equation}
	\kappa_{\infty}=\kappa_0=0 \quad \text{ and } \quad \kappa_{\frac 12}=\mfrac{(-1)^\lambda3}{4}.
\end{equation}

Following Kohnen \cite{kohnen-1,kohnen-newforms} we define an operator $L$ on automorphic functions as follows.
If $f$ satisfies $f\sl_{k}\gamma = \nu(\gamma) f$ for all $\gamma\in \Gamma_0(4)$ then we define
\begin{equation}
	Lf := \frac{1}{2(1+i^{2k})} \sum_{w=0}^3 f\sl_{k} \pMatrix {1+w}{1/4}{4w}1.
\end{equation}
It is not difficult to show that $L$ maps Maass cusp forms to Maass cusp forms.
It follows from \cite{kohnen-1} (see also \cite{katok-sarnak}) that $L$ is self-adjoint, that it commutes with the Hecke operators $T_{p^2}$, and that it satisfies the equation
\begin{equation}
	(L-1)(L+\tfrac 12)=0
\end{equation}
(Kohnen proves this in the holomorphic case, but the necessary modifications are simple).
The space $\mathcal V_k$ decomposes as $\mathcal V_k=\mathcal V_k^+ \oplus \mathcal V_k^-$ where $\mathcal V_k^+$ is the eigenspace with eigenvalue $1$, and $\mathcal V_k^-$ is the eigenspace with eigenvalue $-\frac 12$.
For each $f\in \mathcal V_k$, we have $f\in \mathcal V_k^+$ if and only if $\rho_{f,\infty}(n)=0$ for $(-1)^\lambda n\equiv 2,3\pmod{4}$.
The following lemma describes the action of $L$ on Fourier expansions.

\begin{lemma} \label{lem:L-fourier}
Let $k=\pm \frac 12=\lambda+\frac 12$ and $\nu=\nu_\theta^{2k}$.
Suppose that $f \sl_{k}\gamma = \nu(\gamma) f$ for all $\gamma\in \Gamma$.
For each cusp $\a$ of $\Gamma$ write the Fourier expansion of $f$ as
\begin{equation}
	(f \sl_{k} \sigma_\a)(z) = \sum_{n\in \Z} c_{f,\a}(n,y) e(n_\a x).
\end{equation}
Then
\begin{equation}
	c_{Lf,\infty}(n,y) = 
	\begin{dcases}
		\mfrac 12 c_{f,\infty} (n,y) + \mfrac 1{2(1-i^{2k})} \, c_{f,\a} \left(\mfrac{n}4+\kappa_{\a},4y\right) & \text{ if } (-1)^\lambda n\equiv 0,1\pmod{4}, \\
		-\mfrac 12 c_{f,\infty} (n,y) & \text{ if }(-1)^\lambda n\equiv 2,3 \pmod{4},
	\end{dcases}
\end{equation}
where $\a=0$ if $n\equiv 0 \pmod 4$ and $\a=\frac 12$ if $n\equiv (-1)^\lambda \pmod{4}$.
\end{lemma}
\begin{proof}
Let $A_w = \pmatrix {1+w}{1/4}{4w}1$.
Since $A_2 = \pmatrix 3{-2}8{-5} \pmatrix1{3/4}01$ and $\nu(\pmatrix 3{-2}8{-5})=i^{2k}$ we have
\begin{align}
	f \sl_k A_0 + f \sl_k A_2 
	&= f(z+\tfrac 14) +  i^{2k} f(z+\tfrac 34) \\
	&= (1+i^{2k}) \sum_{(-1)^\lambda n\equiv 0,1(4)} c_{f,\infty}(n,y) e(nx) - (1+i^{2k}) \sum_{(-1)^\lambda n\equiv 2,3(4)} c_{f,\infty}(n,y) e(nx).
\end{align}
For $w=1,3$ we have $A_1 = \pmatrix {-1}1{-4}3 \sigma_{\frac 12} \pmatrix 200{1/2}$ and $A_3 = \pmatrix {-1}1{-4}3 \sigma_0 \pmatrix 200{1/2}$.
Since $\nu(\pmatrix {-1}1{-4}3)=i^{2k}$, we have
\begin{align}
	f \sl_k A_1 + f \sl_k A_3 
	&= i^{2k} (f\sl_k \sigma_0)(4z) + i^{2k} (f\sl_k \sigma_{\frac 12}(4z)) \\
	&= i^{2k}\sum_{n\equiv 0(4)} c_{f,0}(\tfrac n4,4y) e(nx) + i^{2k}\sum_{n\equiv (-1)^\lambda(4)} c_{f,\frac 12}(\tfrac {n}4+\kappa_{\frac 12},4y) e(nx).
\end{align}
The lemma follows.
\end{proof}

The analogue of $L$ for holomorphic cusp forms is defined as follows.
If for some $\ell$ we have $F(\gamma z) = \nu(\gamma)(cz+d)^\ell F(z)$ for all $\gamma\in \Gamma_0(4)$ then $f(z)=y^{\ell/2}F(z)$ satisfies $f\sl_\ell\gamma = \nu(\gamma) f$, and we define
\begin{equation}
	L^* F := y^{-\frac \ell2} L f.
\end{equation}
The plus space $\mathcal S_\ell^+$ of holomorphic cusp forms is defined as the subspace of $\mathcal S_\ell$ consisting of forms $F$ satisfying $L^*F=F$.
If $\rho_{F,\a}(n)$ is the $n$-th coefficient of $F$ at the cusp $\a$ then, in the notation of the previous lemma, we have $\rho_{F,\a}(\frac n4+\kappa_\a,4y)=\frac 12 c_{f,\a}(\frac n4+\kappa_\a,4y)$.
Therefore we have the following analogue of Lemma~\ref{lem:L-fourier}.

\begin{lemma}
Let $k=\pm \frac 12=\lambda+\frac 12$ and $\nu=\nu_\theta^{2k}$.
Suppose that $\ell\equiv k\pmod{2}$ and that $F\in \mathcal S_\ell(\nu)$.
Then
\begin{equation}
	\rho_{L^*F,\infty}(n,y) = 
	\begin{dcases}
		\mfrac 12 \rho_{F,\infty} (n,y) + \mfrac 1{(1-i^{2k})} \, \rho_{F,\a} \left(\mfrac{n}4+\kappa_{\a},4y\right) & \text{ if } (-1)^\lambda n\equiv 0,1\pmod{4}, \\
		-\mfrac 12 \rho_{F,\infty} (n,y) & \text{ if }(-1)^\lambda n\equiv 2,3 \pmod{4},
	\end{dcases}
\end{equation}
where $\a=0$ if $n\equiv 0 \pmod 4$ and $\a=\frac 12$ if $n\equiv (-1)^\lambda \pmod{4}$.
\end{lemma}

To state the plus space version of the Kuznetsov trace formula, we first fix some notation.
Recall that $S_k^+(m,n,c)$ is the plus space Kloosterman sum
\begin{equation}
	S_k^+(m,n,c) = e\Big(\!-\mfrac k4\Big) \sum_{d\bmod c} \pfrac{c}{d} \ep_d^{2k} e\pfrac{m\bar d+n d}{c} \times
	\begin{cases}
		1 & \text{ if } 8\mid c,\\
		2 & \text{ if }4\mid\mid c.
	\end{cases}
\end{equation}
Let $\varphi:[0,\infty)\to\R$ be a smooth test function which satisfies
\begin{gather} \label{eq:varphi-cond}
	\varphi(0) = \varphi'(0) = 0 \quad \text{ and } \quad
	\varphi^{(j)}(x) \ll x^{-2-\varepsilon} \quad \text{ for }j=0,1,2,3.
\end{gather}
Define the integral transforms
\begin{align}
	\tilde \varphi(\ell) &:= \frac 1\pi\int_0^\infty J_{\ell-1}(x) \varphi(x) \, \mfrac{dx}{x}, \label{eq:phi-tilde-def}
	 \\
	\hat\varphi(r) &:= \frac{-i \, \xi_k(r)}{\cosh 2\pi r} \int_0^\infty \left( \cos(\tfrac{\pi k}2+\pi i r)J_{2ir}(x) - \cos(\tfrac{\pi k}2-\pi i r)J_{-2ir}(x) \right) \phi(x) \mfrac{dx}{x}, \label{eq:phi-hat-def}
\end{align}
where
\begin{equation}
	\xi_k(r) := \frac{\pi^2}{\sinh \pi r \Gamma(\frac{1-k}{2}+ir)\Gamma(\frac{1-k}2-ir)} \sim \tfrac{1}{2} \pi r^k \quad \text{ as }r\to\infty.
\end{equation}
Note that $\hat \varphi(r)$ is real-valued when $r\geq 0$ and when $ir\in (-\frac 14,\frac 14)$.
If $d$ is a fundamental discriminant, let $\chi_d = \pfrac d\cdot$ and let $L(s,\chi_d)$ denote the Dirichlet $L$-function with Dirichlet series
\begin{equation}
 	L(s,\chi_d) := \sum_{n=1}^\infty \frac{\chi_d(n)}{n^s}. 
\end{equation} 
Finally, we define
\begin{equation}
	\fS_d(w,s) = \sum_{\ell\mid w} \mu(\ell) \chi_{d}(\ell) \frac{\tau_{s}(w/\ell)}{\sqrt \ell},
\end{equation}
where $\tau_s$ is the normalized sum of divisors function
\begin{equation}
	\tau_s(\ell) = \sum_{ab=\ell} \pfrac ab^s = \frac{\sigma_{2s}(\ell)}{\ell^s}.
\end{equation}

\begin{theorem} \label{thm:KTFplus-space+}
Let $\varphi:[0,\infty)\to\R$ be a smooth test function satisfying \eqref{eq:varphi-cond}.
Let $k=\pm \frac 12=\lambda+\frac 12$ and $\nu=\nu_\theta^{2k}$.
Suppose that $m,n\geq 1$ with $(-1)^\lambda m,(-1)^\lambda n\equiv 0,1\pmod{4}$, and write
\begin{equation}
	(-1)^\lambda m = v^2 d', \quad (-1)^\lambda n = w^2 d, \quad \text{ with }d,d'\text{ fundamental discriminants.}
\end{equation}
Fix an orthonormal basis of Maass cusp forms $\{u_j\}\subset \mathcal V_k^+$ with associated spectral parameters $r_j$ and coefficients $\rho_j(n)$.
For each $\ell\equiv k\pmod{2}$ with $\ell>2$, fix an orthonormal basis of holomorphic cusp forms $\mathcal H_\ell^+\subset \mathcal S_\ell^+$ with normalized coefficients given by
\begin{equation} \label{eq:g-normalize}
	g(z) = \sum_{n=1}^{\infty} (4\pi n)^{\frac{\ell-1}2}\rho_g(n) e(nz) \quad \text{ for } g\in \mathcal H_\ell^+.
\end{equation}
Then
\begin{multline}
	\sum_{0<c\equiv 0(4)} \frac{S_k^+(m,n,c)}{c} \varphi \pfrac{4\pi\sqrt{mn}}{c}  \\
	=
	6\sqrt{mn} \sum_{j\geq 0} \frac{\overline{\rho_{j}(m)}\rho_{j}(n)}{\cosh\pi r_j} \hat\varphi(r_j)
	+
	\mfrac{3}{2}\sum_{\ell\equiv k\bmod 2} e\ptfrac{\ell-k}4 \tilde\varphi(\ell) \Gamma(\ell) \sum_{g\in \mathcal H_\ell^+} \overline{\rho_{g}(m)}\rho_{g}(n)  
	\\
	+ \mfrac 12\int_{-\infty}^\infty \pfrac{d}{d'}^{ir} \frac {L(\frac 12-2ir,\chi_{d'})L(\frac 12+2ir,\chi_{d})\fS_{d'}(v,2ir)\fS_{d}(w,2ir)}{|\zeta(1+4ir)|^2 \cosh \pi r |\Gamma(\frac{k+1}{2}+ir)|^2} \hat\varphi(r) \, dr.
\end{multline}
\end{theorem}

Bir\'o \cite[Theorem~B]{biro} stated a version of Theorem~\ref{thm:KTFplus-space+} for $\Gamma_0(4N)$ in the case $k=\frac 12$ under the added assumption that $\tilde \varphi(\ell)=0$ for all $\ell$.
His theorem involves coefficients of half-integral weight Eisenstein series at cusps instead of Dirichlet $L$-functions.

To prove Theorem~\ref{thm:KTFplus-space+}, we start with Proskurin's version of the Kuznetsov formula \cite{proskurin-new} which is valid for arbitrary weight $k$ and for the cusp-pair $\infty\infty$.
The necessary modifications for an arbitrary cusp-pair are straightforward (see \cite{deshouillers-iwaniec} for details in the $k=0$ case).
Recall the definitions of the generalized Kloosterman sum $S_{\a\b}(m,n,c,\nu)$ in \eqref{eq:S-ab-ref} and the Eisenstein series coefficients $\phi_{\a\b}(m,s)$ in \eqref{eq:phi-ab-def}.

\begin{proposition} \label{prop:proskurin}
Suppose that $\varphi$ satisfies \eqref{eq:varphi-cond}.
Suppose that $m,n\geq 1$ and that $k=\pm \frac 12$.
Let $\nu=\nu_\theta^{2k}$ and $\Gamma=\Gamma_0(4)$ and let $\a,\b$ be cusps of $\Gamma$.
Let $\{u_j\}$ denote an orthonormal basis of Maass cusp forms of weight $k$ with spectral parameters $r_j$.
For each $2< \ell\equiv k\pmod{2}$, let $\mathcal H_\ell$ denote an orthonormal basis of holomorphic cusp forms of weight~$\ell$ with coefficients normalized as in \eqref{eq:g-normalize}.
Then
\begin{multline} \label{eq:ktf-usual}
	e(-\tfrac k4)\sum_{c\in \mathcal C(\a,\b)} \frac{S_{\a\b}(m,n,c,\nu)}{c} \varphi \pfrac{4\pi\sqrt{m_\a n_\b}}{c}
	\\
	 = 4\sqrt{m_\a n_\b} \sum_{j\geq 0} \frac{\overline{\rho_{j\a}(m)}\rho_{j\b}(n)}{\cosh\pi r_j} \hat\varphi(r_j) 
	 + \sum_{\ell\equiv k\bmod 2} e\ptfrac{\ell-k}4 \tilde\varphi(\ell) \Gamma(\ell) \sum_{g\in \mathcal H_\ell} \overline{\rho_{g\a}(m)}\rho_{g\b}(n)
	 \\+
	  \sum_{\c\in\{0,\infty\}} \int_{-\infty}^\infty \pfrac{n_\b}{m_\a}^{ir} \frac{\overline{\phi_{\c\a}(m,\frac 12+ir)} \,\phi_{\c\b}(n,\frac 12+ir)}{\cosh\pi r |\Gamma(\frac{k+1}{2}+ir)|^2} \hat\varphi(r) \, dr.
\end{multline}
\end{proposition}

We will apply \eqref{eq:ktf-usual} for the cusp-pairs $\infty\infty$, $\infty0$, and $\infty\frac 12$, and take a certain linear combination which annihilates all but the plus space coefficients.
The following lemma is essential to make this work.

\begin{lemma} \label{lem:kloo-cusps}
Suppose that $4\mid\mid c$.
Let $k=\pm \frac 12=\lambda+\frac 12$ and $\nu=\nu_\theta^{2k}$. 
Let $\a=0$ or $\frac 12$ according to $(-1)^\lambda n\equiv0,1\pmod{4}$, respectively.
Then
\begin{equation}
	S_{\infty\infty}(m,n,c,\nu) = (1+i^{2k}) S_{\infty\a}(m,\tfrac n4+\kappa_\a,\tfrac c2,\nu).
\end{equation}
\end{lemma}

\begin{proof}
Since $\overline{S_{\a\b}(m,n,c,\nu)} = S_{\a\b}(-m,-n,c,\bar\nu)$, it is enough to show that
\begin{equation} \label{eq:kloo-cusps}
	S_{\infty\infty}(m,n,c,\nu_\theta) = (1+i) \times
	\begin{dcases}
		S_{\infty0}(m,\tfrac n4,\tfrac c2,\nu_\theta) & \text{ if }n\equiv 0\pmod{4}, \\
		S_{\infty\frac 12}(m,\tfrac {n+3}4,\tfrac c2,\nu_\theta) & \text{ if }n\equiv 1\pmod{4}.
	\end{dcases}
\end{equation}
This is proved in \cite[Lemma~A.7]{biro}.
Note that Biro chooses different representatives and scaling matrices for the cusps $0$ and $\frac 12$, which has the effect of changing the factor $(1-i)$ to $(1+i)$.
\end{proof}

\begin{proof}[Proof of Theorem~\ref{thm:KTFplus-space+}]
Let $k,\nu$, and $\a$ be as in Lemma~\ref{lem:kloo-cusps}.
From that lemma and the definition \eqref{eq:sk+def} it follows that
\begin{equation}
	S_k^+(m,n,c) = e(-\tfrac k4) S_{\infty\infty}(m,n,c,\nu) + \delta_{4\mid\mid c} \sqrt 2 \, S_{\infty \a}(m,\tfrac n4+\kappa_\a,\tfrac c2,\nu).
\end{equation}
Therefore
\begin{multline} \label{eq:S+-split}
	\sum_{4\mid c>0} \frac{S^+_k(m,n,c)}{c} \varphi\pfrac{4\pi\sqrt{mn}}{c}
	= e(-\tfrac k4) \sum_{4\mid c>0} \frac{S_{\infty\infty}(m,n,c,\nu)}{c} \varphi\pfrac{4\pi\sqrt{mn}}{c} \\
	+ \mfrac1{\sqrt 2}\sum_{2\mid\mid c>0} \frac{S_{\infty\a}(m,\frac n4+\kappa_\a,c,\nu)}{c} \varphi\pfrac{4\pi\sqrt{m(\frac n4+\kappa_\a)_\a}}{c}.
\end{multline}
Note that $\mathcal C(\infty,\a)=\{c\in \Z_+:c\equiv 2\pmod{4}\}$ for $\a=0,\frac 12$.
We apply Proposition~\ref{prop:proskurin} for each of the cusp-pairs $\infty\infty$ and $\infty\a$ on the right-hand side of \eqref{eq:S+-split}.
We fix an orthonormal basis $\{u_j^+\}$ for $\mathcal V_k^+$ and we choose an orthonormal basis $\{u_j\}$ for $\mathcal V_k$ such that $\{u_j^+\} \subseteq \{u_j\}$.
Then we do the same for $\mathcal H_k^+\subseteq \mathcal H_k$.
The Maass form contribution is
\begin{equation}
	4\sqrt{mn}  \sum_{u_j\in \mathcal V_k} \frac{\bar\rho_j(m)}{\cosh\pi r_j} \hat\varphi(r_j) \left( \rho_{j,\infty}(n) + \mfrac{1}{2(1-i^{2k})} \rho_{j,\a}(\tfrac n4+\kappa_\a) \right).
\end{equation}
Let $\rho_j^{(L)}$ denote the coefficients of $L u_j$.
Then by Lemma~\ref{lem:L-fourier} we have
\begin{equation}
	\rho_{j,\infty}(n) + \mfrac{1}{2(1-i^{2k})} \rho_{j,\a}(\tfrac n4+\kappa_\a) = 
	\mfrac 12 \rho_{j,\infty}(n) + \rho^{(L)}_{j,\infty}(n) = \rho_j(n) \times
	\begin{cases}
		\frac 32 & \text{ if }u_j \in \mathcal V_k^+, \\
		0 & \text{ if }u_j \in \mathcal V_k^-.
	\end{cases}
\end{equation}
We compute the contribution from the holomorphic forms similarly.
For the Eisenstein series contribution we apply the following proposition, together with the relation $S_{\a\b}(m,n,c,\nu) = \overline{S_{\a\b}(-m,-n,c,\bar\nu)}$.
\end{proof}

\begin{proposition} \label{prop:eis-lin-comb-r}
Let $k=\frac 12$ and $\nu=\nu_\theta$ and suppose that $m,n\equiv 0,1\pmod{4}$.
Write $m=v^2d'$ and $n=w^2d$, where $d',d$ are fundamental discriminants.
Let $\a=0$ or $\frac 12$ according to $n\equiv0,1\pmod{4}$, respectively.
Then
\begin{multline} \label{eq:eis-lin-comb-r}
	\sum_{\c \in \{\infty,0\}} \overline\phi_{\c\infty}(m,\tfrac 12+ir) \left( \phi_{\c\infty}(n,\tfrac 12+ir) + \mfrac{1+i}{2\cdot 4^{ir}} \phi_{\c\a}\Big(\mfrac n4+\kappa_\a,\tfrac 12+ir\Big) \right)
	 \\= \frac {L(\frac 12-2ir,\chi_{d'})L(\frac 12+2ir,\chi_{d})}{2|\zeta(1+4ir)|^2} \, \pfrac{v}{w}^{2ir} \fS_{d'}(v,2ir)\fS_d(w,2ir).
\end{multline}
\end{proposition}

The proof of this proposition is quite technical, and we will proceed in several steps.
In order to work in the region of absolute convergence, we will evalute the sum
\begin{equation}
	\sum_{\c \in \{\infty,0\}} \overline\phi_{\c\infty}(m,s) \left( \phi_{\c\infty}(n,s) + \mfrac{1+i}{4^{s}} \phi_{\c\a}\Big(\mfrac n4+\kappa_\a,s\Big) \right),
\end{equation}
for $\re(s)$ sufficiently large.
Then, by analytic continuation, we can set $s=\frac 12+ir$ to obtain \eqref{eq:eis-lin-comb-r}.
First, for the term $\c=\infty$, by Lemma~\ref{lem:kloo-cusps} we have
\begin{equation} \label{eq:eis-inf-lin-comb}
	\phi_{\infty\infty}(n,s) + \mfrac{1+i}{4^s} \phi_{\infty\a}\Big(\mfrac n4+\kappa_\a,s\Big) = e\pmfrac 18 \phi^+(n,s),
\end{equation}
where
\begin{equation} \label{eq:def-phi^+}
	\phi^+(n,s) = \sum_{4\mid c>0} \frac{S^+(0,n,c)}{c^{2s}}.
\end{equation}
Here we have written $S^+(m,n,c)=S^+_{1/2}(m,n,c)$ for convenience.
The following proposition evaluates $\phi^+(n,s)$.
It is proved in \cite{ibukiyama-saito} and applied in \cite[Lemma~4]{DIT-cycle}; here we give an alternative proof which uses Kohnen's identity \eqref{eq:kohnen-identity}.

\begin{proposition} \label{prop:phi+eval}
Let $w\in \Z_+$ and let $d$ be a fundamental discriminant. Then
\begin{equation}
	\phi^+(w^2d,s) = 2^{\frac 32-4s}w^{1-2s} \frac{L(2s-\frac 12,\chi_{d})}{\zeta(4s-1)} \fS_d(w,2s-1).
\end{equation}
\end{proposition}

\begin{proof}
By M\"obius inversion, it suffices to prove that
\begin{equation}
	\sum_{\ell\mid w} \chi_{d}(\ell) \ell^{\frac 12-2s} \phi^+\Big(\mfrac{w^2}{\ell^2}d,s\Big) = 2^{\frac 32-4s} w^{1-2s} \tau_{2s-1}(w) \frac{L(2s-\frac 12,\chi_{d})}{\zeta(4s-1)}.
\end{equation}
Writing $\phi^+$ as the Dirichlet series \eqref{eq:def-phi^+}, reversing the order of summation, and applying the identity \eqref{eq:kohnen-identity}, we find that
\begin{align*}
	\sum_{\ell\mid w} \chi_{d}(\ell) \ell^{\frac 12-2s} \phi^+\Big(\mfrac{w^2}{\ell^2}d,s\Big) 
	&= \mfrac{1}{\sqrt 2} \sum_{4\mid c>0} \frac{1}{c^{2s-\frac 12}} \sum_{\ell\mid (w,\frac c4)} \chi_{d}(\ell) \sqrt{\mfrac{2\ell}c} \, S^+\Big(0,\mfrac{w^2}{\ell^2}d;\mfrac c\ell\Big) \\*
	&= 2^{\frac 12-4s} \sum_{c=1}^\infty \frac{T_{w}(0,d;4c)}{c^{2s-\frac 12}}.
\end{align*}
To evaluate $T_w(0,d;4c)$ for a given $c$, we write $4c=tu$, where
\begin{equation}
	u=\prod_{p^a\mid\mid 4c} p^{\lceil \frac a2 \rceil} \quad \text{ and } \quad t=\prod_{p^a\mid\mid 4c} p^{\lfloor \frac a2 \rfloor}.
\end{equation}
Then $b^2\equiv 0\pmod{4c}$ if and only if $b=xu$ for some $x$ modulo $t$.
For each such $b$, let $g=(x,\frac t2)$ and choose $\lambda\in \Z$ such that
\begin{equation}
	\gamma = \pMatrix{t/2g}{x/g}{\lambda}{\frac{1+\lambda x/g}{t/2g}} \in \SL_2(\Z).
\end{equation}
Then $\gamma[c,b,b^2/4c] = [u g^2/t, 0, 0]$ and $\chi_d([c,b,b^2/4c]) = \chi_d(ug^2/t)$.
It follows that
\begin{equation}
	T_w(0,d;4c) = 2\chi_d(u/t) \sum_{\substack{x\bmod t/2 \\ (x,t/2,d)=1}} e\pmfrac{mx}{t/2} =: 2f(c).
\end{equation}
It is straightforward to verify that $f(c)$ is a multiplicative function and that for each prime $p$ we have
\begin{align}
	\text{ if }p\mid d\text{ then } f(p^a) &= 
	\begin{cases}
		c_{p^{\frac a2}}(w) & \text{ if $a$ is even}, \\
		0 & \text{ if $a$ is odd},
	\end{cases}
	\\
	\text{ if }p\nmid d\text{ then } f(p^a) &= \chi_d(p)^a \times 
	\begin{cases}
		p^{\lfloor \frac a2 \rfloor} & \text{ if }p^{\lfloor \frac a2 \rfloor} \mid w, \\
		0 & \text{ otherwise}.
	\end{cases}
\end{align}
Here $c_q(w)$ is the Ramanujan sum which satisfies
\begin{equation}
	\frac{w^{1-s}\sigma_{s-1}(w)}{\zeta(s)} = \sum_{q=1}^\infty \frac{c_q(w)}{q^s} = \prod_p \sum_{a=0}^\infty \frac{c_{p^a}(w)}{p^s}.
\end{equation}
It follows that
\begin{equation}
	\sum_{c=1}^\infty \frac{f(c)}{c^{2s-\frac 12}} = w^{2-4s}\sigma_{4s-2}(w) \frac{L(2s-\frac 12,\chi_d)}{\zeta(4s-1)}.
\end{equation}
The proposition follows.
\end{proof}

Next we evaluate the term in \eqref{eq:eis-lin-comb-r} corresponding to the cusp $\c=0$.
The following lemma will be useful.

\begin{lemma}	\label{lem:0a-kloo-eval}
Let $k=\frac 12$ and $\nu=\nu_\theta$ and suppose that $n\equiv 0,1\pmod{4}$.
Suppose that $4\mid c$ and $\a=0$ or $2\mid\mid c$ and $\a=\frac 12$ according to whether $n\equiv0$ or $1\pmod{4}$, respectively.
Then
\begin{equation} \label{eq:0a-kloo-eval}
	S_{0\a}(0,\tfrac n4+\kappa_\a,c,\nu_\theta) = \tfrac 14 S_{\infty\infty}(0,n,4c,\nu_\theta).
\end{equation}
\end{lemma}

\begin{proof}
For each cusp $\a$ we have $(\frac n4+\kappa_\a)_\a = \frac n4$.
Suppose first that $n\equiv 0\pmod{4}$ and $\a=0$.
A straightforward computation shows that $S_{00}(m,n,c,\nu) = S_{\infty\infty}(m,n,c,\nu)$ for all $m,n\in \Z$.
From the definition of $S_{\infty\infty}(m,n,c,\nu)$ it follows that, for $c\equiv 0\pmod{4}$, we have $S_{00}(0,\frac n4,c,\nu) = \frac 14S_{\infty\infty}(0,n,4c,\nu)$.

Now suppose that $n\equiv 1\pmod{4}$ and $\a=\frac 12$.
We will prove \eqref{eq:0a-kloo-eval} directly from the definition of $S_{0\frac 12}(m,n,c,\nu)$.
Let $\pmatrix abcd = \sigma_0^{-1}\pmatrix ABCD \sigma_{\frac 12}$, where $\pmatrix ABCD\in \Gamma_0(4)$.
Then $2\mid\mid c$ and $a,d$ are odd, so (after shifting by $\pmatrix 1*01$ on the right) we can assume that $4\mid b$.
Then $\ep_D = \ep_{a+2b} = \ep_a = \ep_d$
since $ad\equiv 1\pmod{4}$.
We also have
\begin{equation}
	\pmfrac CD = \pmfrac{-4b}{a+2b} = \pmfrac{2a}{a+2b} = (-1)^{\frac{a-1}2} \pmfrac 2a \pmfrac{2b}a = \pmfrac{4c}{a} = \pmfrac{4c}{d}
\end{equation}
since $bc\equiv -1\pmod{a}$ and $ad\equiv 1\pmod{4c}$.
It follows that
\begin{equation}
	S_{0\a}(0,\tfrac n4+\kappa_\a,c,\nu_\theta) = \sum_{d\bmod c} \pmfrac{4c}d \ep_d \, e\pmfrac{nd}{4c}.
\end{equation}
Note that replacing $d$ by $d+c$ has no net effect since $\ep_{d+c}=\ep_{-d}$ and $\pfrac{4c}{d+c} = -\pfrac{-1}{d} \pfrac cd$, so
\begin{equation}
	\pmfrac{4c}{d+c} \ep_{d+c} \, e\pmfrac{n(d+c)}{4c} = \pmfrac{4c}d e\pmfrac{nd}{4c} \left[-\ep_{-d}\pmfrac{-1}d e\pmfrac{n}{4} \right] = \pmfrac{4c}d \ep_d \, e\pmfrac{nd}{4c}
\end{equation}
since $n\equiv 1\pmod{4}$.
The relation~\eqref{eq:0a-kloo-eval} follows.
\end{proof}

\begin{proposition} \label{prop:eis-0-lin-comb}
Let $k=\frac 12$ and $\nu=\nu_\theta$ and suppose that $n\equiv 0,1\pmod{4}$.
Write $n=w^2d$ with $d$ a fundamental discriminant.
Let $\a=0$ or $\frac 12$ according to $n\equiv0,1\pmod{4}$, respectively.
Then
\begin{equation} \label{eq:eis-0a-eval}
	\phi_{0\infty}(n,s) + \mfrac{1+i}{4^s} \, \phi_{0\a}\left(\mfrac n4+\kappa_\a,s\right) = \mfrac i{4^s} w^{1-2s} \frac{L(2s-\frac 12,\chi_{d})}{\zeta(4s-1)} \fS_d(w,2s-1).
\end{equation}
\end{proposition}

\begin{proof}
We will prove that
\begin{equation}
	\phi_{0\infty}(n,s) + \mfrac{1+i}{4^s} \, \phi_{0\a}\left(\mfrac n4+\kappa_\a,s\right) =
	 i\cdot 2^{2s-\frac 32} \phi^+(n,s);
\end{equation}
then equation \eqref{eq:eis-0a-eval} will follow from Proposition~\ref{prop:phi+eval}.
A straightforward computation gives the relation $S_{0\infty}(m,n,c,\nu_\theta)=iS_{\infty0}(m,n,c,\nu_\theta)$.
This, together with Lemma~\ref{lem:kloo-cusps}, shows that
\begin{equation} \label{eq:phi0infty}
	\phi_{0\infty}(n,s) = \mfrac{2^{2s}}{1-i} \sum_{4\mid\mid c>0} \frac{S_{\infty\infty}(0,n,c,\nu_\theta)}{c^{2s}}.
\end{equation}
Next, by Lemma~\ref{lem:0a-kloo-eval} we find that
\begin{equation} \label{eq:phi0a}
	\mfrac{1+i}{4^s} \, \phi_{0\a}\left(\mfrac n4+\kappa_\a,s\right) = \mfrac{2^{2s}}{2(1-i)} \sum \frac{S_{\infty\infty}(0,n,c,\nu_\theta)}{c^{2s}},
\end{equation}
where the sum is over $c\equiv 0\pmod{16}$ if $\a=0$, or $c\equiv 8\pmod{16}$ if $\a=\frac 12$.
We claim that we can let the sum run over all $c\equiv 0\pmod{8}$ in either case.
Equivalently,
\begin{equation} \label{eq:kloo=0}
	S_{\infty\infty}(0,n,c,\nu_\theta)=0 \quad \text{ when } \quad
	\begin{cases}
		c\equiv 8\pmod{16} & \text{ if }n\equiv 0\pmod{4}, \\
		c\equiv 0\pmod{16} & \text{ if }n\equiv 1\pmod{4}.
	\end{cases}
\end{equation}

To see this, we decompose the Kloosterman sum as follows (see Lemma~1 of \cite{sturm}): if $c=2^t c'$ with $c'$ odd, then
\begin{equation}
	S_{\infty\infty}(0,n,c,\nu_\theta) = \ep_{c'}^{-1} G(n,c') \sum_{r\bmod 2^t} \pmfrac{2^t}{r} \ep_r e\pmfrac{nr}{2^t},
\end{equation}
where $G(n,c')$ is a Gauss sum.
In the case $n\equiv 0\pmod{4}$ and $c\equiv 8\pmod{16}$ it is easy to see that
\begin{equation}
	\sum_{r\bmod 8} \pmfrac{8}{r} \ep_r e\pmfrac{nr}{8} = 0.
\end{equation}
If $n\equiv 1\pmod{4}$ then, by replacing $r$ by $r+2^{t-2}$, we see that
\begin{equation}
	\sum_{r\bmod 2^t} \pmfrac{2^t}{r} \ep_r e\pmfrac{nr}{2^t} = e\pmfrac{n}{4} \sum_{r\bmod 2^t} \pmfrac{2^t}{r} \ep_r e\pmfrac{nr}{2^t}
\end{equation}
as long as $t\geq 4$, from which it follows that the sum modulo $2^t$ is zero.

By \eqref{eq:phi0infty}, \eqref{eq:phi0a}, and \eqref{eq:kloo=0}, we conclude that
\begin{align}
	\phi_{0\infty}(n,s) + \mfrac{1+i}{4^s} \, \phi_{0\a}\left(\mfrac n4+\kappa_\a,s\right) 
	&= \mfrac{2^{2s-1}}{1-i} \left( 2\sum_{4\mid\mid c>0} \frac{S_{\infty\infty}(0,n,c,\nu_\theta)}{c^{2s}} + \sum_{8\mid c>0} \frac{S_{\infty\infty}(0,n,c,\nu_\theta)}{c^{2s}} \right) \\
	&= i\cdot 2^{2s-\frac 32} \phi^+(n,s),
\end{align}
which completes the proof of the proposition.
\end{proof}

\begin{proof}[Proof of Proposition~\ref{prop:eis-lin-comb-r}]

By equation \eqref{eq:eis-inf-lin-comb} and Propositions~\ref{prop:phi+eval} and \ref{prop:eis-0-lin-comb} we have
\begin{multline}
	\sum_{\c \in \{\infty,0\}} \overline\phi_{\c\infty}(m,s) \left( \phi_{\c\infty}(n,s) + \mfrac{1+i}{4^{s}} \phi_{\c\a}\Big(\mfrac n4+\kappa_\a,s\Big) \right) \\
	 = \left(e\pmfrac 18 2^{\frac 32-4s}\overline\phi_{\infty\infty}(m,s) + i\cdot 2^{-2s}\overline\phi_{0\infty}(m,s)  \right) w^{1-2s} \frac{L(2s-\frac 12,\chi_{d})}{\zeta(4s-1)} \fS(w^2d,2s-1).
\end{multline}
Then by \eqref{eq:phi0infty} we have (writing $s=\sigma+ir$)
\begin{align}
	e\pmfrac 18 2^{\frac 32-4s} & \overline\phi_{\infty\infty}(m,s) + i\cdot 2^{-2s}\overline\phi_{0\infty}(m,s) \\
	&= (1+i) 2^{1-4s} \sum_{4\mid c>0} \frac{S_{\infty\infty}(0,m,c,\nu_\theta)}{c^{2\bar s}} + \mfrac{4^{\bar s-s}}{1-i} \sum_{4\mid\mid c>0} \frac{S_{\infty\infty}(0,m,c,\nu_\theta)}{c^{2\bar s}} \\
	&= \mfrac{2^{-4ir}}{1-i} \Big(\phi^+(m,\bar s) + (4^{1-2\sigma}-1) \phi_{\infty\infty}(m,\bar s) \Big).
\end{align}
The proposition follows after applying Proposition~\ref{prop:phi+eval} and setting $s=\frac 12+ir$, noting that the factor $4^{1-2\sigma}-1$ in the second term vanishes.
\end{proof}

\section{Proof of Theorem~\ref{thm:main-kloo}} \label{sec:proof}

Let $a=4\pi\sqrt{mn}$ and $x>0$ and let $x^{\frac 13}\ll T \ll x^{\frac 23}$ be a free parameter to be chosen later.
We choose a test function $\varphi=\varphi_{a,x,T}:[0,\infty)\to[0,1]$ satisfying
\begin{enumerate}[(i)] \setlength\itemsep{.5em}
	\item $\varphi(t)=1$ for $\mfrac{a}{2x}\leq t\leq \mfrac ax$,
	\item $\varphi(t)=0$ for $t\leq \mfrac{a}{2x+2T}$ and $t\geq \mfrac{a}{x-T}$,
	\item $\varphi'(t) \ll \left( \mfrac{a}{x-T} - \mfrac ax \right)^{-1} \ll \mfrac{x^2}{aT}$, and
	\item $\varphi$ and $\varphi'$ are piecewise monotonic on a fixed number of intervals (whose number is independent of $a,x,T$).
\end{enumerate}
We apply the plus space Kuznetsov formula in Theorem~\ref{thm:KTFplus-space+} with this test function and we estimate each of the terms on the right-hand side.

We begin by estimating the contribution from the holomorphic cusp forms
\begin{equation} \label{eq:K-h}
	\mathcal K^{h} := \sum_{\ell\equiv k\bmod 2} e\ptfrac{\ell-k}4 \tilde\varphi(\ell) \Gamma(\ell) \sum_{g\in \mathcal H_\ell^+} \overline{\rho_{g}(m)}\rho_{g}(n).
\end{equation}
Since the operator $L$ commutes with the Hecke operators we may assume that the orthonormal basis $\mathcal H_\ell^+$ is also a basis consisting of Hecke eigenforms.
We will estimate $\mathcal K^h$ by applying the Kohnen-Zagier formula \cite{kohnen-zagier} and Young's hybrid subconvexity bound \cite{young-subconvexity}.
Let $g \in \mathcal H_\ell^+$ and recall that the coefficients of $g$ are normalized so that
\begin{equation}
	g(z) = \sum_{n=1}^{\infty} (4\pi n)^{\frac{\ell-1}2}\rho_g(n) e(nz).
\end{equation}
Since we are working in the plus space, the Shimura correspondence is an isomorphism between $\mathcal S_\ell^+(\nu)$ and the space $\mathcal S_{2\ell-1}$ of (even) weight $2\ell-1$ cusp forms on $\Gamma_1$.
So $g$ lifts to a unique normalized $f \in S_{2\ell-1}$ with Fourier expansion
\begin{equation}
	f(z) = \sum_{n=1}^\infty n^{\ell-1}a_f(n) e(nz), \quad \text{ where } a_f(1)=1.
\end{equation}
The coefficients $\rho_g$ and $a_f$ are related via
\begin{equation} \label{eq:g-f-shimura}
	\rho_g(v^2|d|) = \rho_g(|d|) \sum_{u\mid v} \mu(u) \pmfrac du u^{-\frac 12} a_f(v/u),
\end{equation}
where $d$ is a fundamental discriminant with $(-1)^\lambda d>0$.
Using Deligne's bound $|a_f(n)|\leq \sigma_0(n)$, it follows that
\begin{equation} \label{eq:rho-g-non-sqfree}
	|\rho_g(v^2|d|)| \leq |\rho_g(|d|)| \sigma_0^2(v).
\end{equation}

Suppose that $g$ is normalized so that $\langle g,g \rangle=1$.
If $d$ is a fundamental discriminant satisfying $(-1)^\lambda d>0$ then the Kohnen-Zagier formula \cite[Theorem~1]{kohnen-zagier} can be written as
\begin{equation}
	\Gamma(\ell) |\rho_g(|d|)|^2 = 4\pi \frac{\Gamma(2\ell-1)}{(4\pi)^{2\ell-1}\langle f,f \rangle} L(\tfrac 12,f\times \chi_d),
\end{equation}
where $L(s,f\times \chi_d)$ is the twisted $L$-function with Dirichlet series
\begin{equation} \label{eq:Lfs-def}
	L(s,f\times \chi_d) = \sum_{m=1}^\infty \frac{a_f(m) \chi_d(m)}{m^s}.
\end{equation}
By a result of Hoffstein and Lockhart (see \cite[Corollary~0.3]{hoffstein-lockhart} and the second remark that follows it, and note that their normalization differs from ours)
we have the bound
\begin{equation}
	\frac{\Gamma(2\ell-1)}{(4\pi)^{2\ell-1}\langle f,f \rangle} \ll \ell^\ep,
\end{equation}
so we conclude that
\begin{equation}
	\Gamma(\ell) |\rho_g(|d|)|^2 \ll L(\tfrac 12, f\times \chi_d) \ell^\ep.
\end{equation}
Let $\mathcal H_{2\ell-1}$ be the image in $\mathcal S_{2\ell-1}$ of the Shimura lift of $\mathcal H_\ell^+(\nu)$.
Young's hybrid subconvexity bound \cite[Theorem~1.1]{young-subconvexity} yields
\begin{equation}
	\sum_{f\in \mathcal H_{2\ell-1}} L(\tfrac 12, f\times \chi_d)^3 \ll (\ell d)^{1+\varepsilon}
\end{equation}
for odd fundamental $d$.
See Appendix~\ref{sec:appendix} for the case of even fundamental discriminants $d$.
Applying H\"older's inequality in the case $\frac 16+\frac 16+\frac 23=1$, together with the fact that $\# \mathcal H_{2\ell-1}\asymp \ell$, we obtain the following theorem for $d,d'$ fundamental discriminants.
It is extended to all $m,n$ using \eqref{eq:rho-g-non-sqfree}.

\begin{theorem} \label{thm:young-hol-coeff}
Let $\ell \equiv k \pmod{2}$ with $k=\pm \frac 12=\lambda+\frac 12$ and
suppose that $\mathcal H_{\ell}^+$ is an orthonormal basis for $\mathcal S_\ell^+$ consisting of Hecke eigenforms.
Suppose that $m,n$ are integers with $(-1)^\lambda m, (-1)^\lambda n >0$, and write $(-1)^\lambda m=v^2 d'$ and $(-1)^\lambda n=w^2 d$ with $d,d'$ fundamental discriminants.
Then
\begin{equation}
	\Gamma(\ell) \sum_{g\in \mathcal H_{\ell}^+} |\rho_g(|m|)\rho_g(|n|)| \ll \ell |dd'|^{\frac 16+\varepsilon}(vw)^\varepsilon.
\end{equation}
\end{theorem}

Applying Theorem~\ref{thm:young-hol-coeff} to the sum \eqref{eq:K-h} we find that
\begin{equation}
	\mathcal K^h \ll |dd'|^{\frac 16+\varepsilon} (vw)^\varepsilon \sum_{\ell\equiv k(2)} \ell \, \tilde\varphi(\ell).
\end{equation}
The latter sum was estimated in \cite{sarnak-tsimerman} (see the discussion following (50); see also Lemma~5.1 of \cite{dunn}) where the authors found that $\sum_\ell \ell \, \tilde\varphi(\ell) \ll \sqrt{mn}/x$.
We conclude that
\begin{equation} \label{eq:K-h-est}
	\mathcal K^h \ll \frac{vw|dd'|^{\frac 23}}{x}(mn)^\varepsilon.
\end{equation}

Next, we estimate the contribution from the Maass cusp forms
\begin{equation}
	\mathcal K^m := \sqrt{mn} \sum_{j\geq 0} \frac{\overline{\rho_{j}(m)}\rho_{j}(n)}{\cosh\pi r_j} \hat\varphi(r_j).
\end{equation}
We follow the same general idea as in the holomorphic case, but instead of the Kohnen-Zagier formula we apply a formula of Baruch and Mao \cite{baruch-mao}.
As in the holomorphic case, we may assume that the orthonormal basis $\{u_j\}$ of $\mathcal V_k^+$ consists of eigenforms for the Hecke operators. 
Suppose that $u_j \in \mathcal V_k^+$ has spectral parameter $r_j$.
The lowest eigenvalue is $\lambda_0=\frac 3{16}$ which corresponds to $u_0=y^{1/4}\theta(z)$ or its conjugate.
Since the coefficients $\rho_0(n)$ are supported on squares and since $m,n$ are not both squares, we find that the term in $\mathcal K^m$ corresponding to $j=0$ does not appear.
In what follows we assume that $j\geq 1$.

Theorem~1.2 of \cite{baruch-mao} shows that there is a unique normalized Maass cusp form $v_j$ of weight~$0$ with spectral parameter $2r_j$ which is even if $k=\frac 12$ and odd if $k=-\frac 12$, and such that the Hecke eigenvalues of $u_j$ and $v_j$ agree.
Since there are no exceptional eigenvalues for weight~$0$ on $\SL_2(\Z)$ this lift implies that there are no exceptional eigenvalues in weights $\pm \frac 12$ in the plus space.
It follows that $r_j\geq 0$ for each $j\geq 1$ (in fact $r_1\approx 1.5$).
If $a_j(n)$ is the $n$-th coefficient of $v_j$ (with respect to the Whittaker function, not the $K$-Bessel function) then for $d$ a fundamental discriminant we have
\begin{equation}
	w \, \rho_j(dw^2) = \rho_j(d) \sum_{\ell\mid w} \ell^{-1}\mu(\ell)\chi_d(\ell) a_j(w/\ell).
\end{equation}
Let $\theta$ denote an admissible exponent toward the Ramanujan conjecture in weight $0$; we have $\theta\leq \frac 7{64}$ by work of Kim and Sarnak \cite{kim-sarnak}.
Then $a_j(w)\ll w^{\theta+\ep}$ since $v_j$ is normalized so that $a_j(1)=1$. 
It follows that
\begin{equation}
	w |\rho_j(dw^2)| \ll w^{\theta+\ep} |\rho_j(d)|.
\end{equation}

Suppose that $d$ is a fundamental discriminant and that $\langle u_j, u_j \rangle =1$.
Then Theorem~1.4 of \cite{baruch-mao} implies that
\begin{equation}
	|\rho_j(d)|^2 = \frac{L(\frac 12, v_j \times \chi_d)}{\pi |d|\langle v_j,v_j \rangle} \left|\Gamma\left(\mfrac{1-k\sgn d}2 - ir_j\right)\right|^2,
\end{equation}
where $L(\frac 12, v_j\times \chi_d)$ is defined in a similar way as \eqref{eq:Lfs-def}.
Hoffstein and Lockhart \cite[Corollary~0.3]{hoffstein-lockhart} proved that $\langle v_j,v_j \rangle^{-1} \ll (1+r_j)^\varepsilon e^{2\pi r_j}$ (again, note that the Fourier coefficients are normalized differently in that paper).
It follows that
\begin{equation}
	|d| \sum_{r_j\leq x} \frac{|\rho_j(d)|^2}{\cosh\pi r_j} \ll \sum_{2r_j\leq 2x} (1+r_j)^{-k\sgn(d)+\varepsilon}  L(\tfrac 12,v_j\times \chi_d).
\end{equation}
Young's subconvexity result \cite[Theorem~1.1]{young-subconvexity} in this case shows that
\begin{equation}
	\sum_{T\leq r_j \leq T+1} L(\tfrac 12, v_j \times \chi_d)^3 \ll (|d|(1+T))^{1+\varepsilon}.
\end{equation}
After applying H\"older's inequality as above, we obtain the following.

\begin{theorem} \label{thm:maass-young-est}
Let $k=\pm \frac 12=\lambda+\frac 12$.
Suppose that $\{u_j\}$ is an orthonormal basis for $\mathcal V_k^+$ consisting of Hecke eigenforms with spectral parameters $r_j$ and coefficients $\rho_j$.
Suppose that $m,n$ are integers with $(-1)^\lambda m, (-1)^\lambda n >0$, and write $(-1)^\lambda m=v^2 d'$ and $(-1)^\lambda n=w^2 d$ with $d,d'$ fundamental discriminants not both equal to $1$.
Then
\begin{equation}
	\sqrt{|mn|} \sum_{r_j \leq x} \frac{|\rho_j(m)\rho_j(n)|}{\cosh\pi r_j} \ll |dd'|^{\frac 16} (vw)^{\theta} x^{2-\frac 12 k (\sgn m+\sgn n)} (mnx)^\varepsilon.
\end{equation}
\end{theorem}

To estimate $\mathcal K^m$ we consider the dyadic sums
\begin{equation}
	\mathcal K^m(A) := \sqrt{mn} \sum_{A\leq r_j < 2A} \frac{\bar{\rho_j(m)}\rho_j(n)}{\cosh\pi r_j} \hat\varphi(r_j)
\end{equation}
for $A\geq 1$.
Theorem~\ref{thm:maass-young-est} gives one estimate for the coefficients $|\rho_j(m)\rho_j(n)|$.
Applying Cauchy-Schwarz and Theorem~4.1 with $\beta=\frac 12+\varepsilon$ we obtain a second estimate:
\begin{equation}
  	\sqrt{mn} \sum_{r_j \leq A} \frac{|\rho_j(m)\rho_j(n)|}{\cosh\pi r_j} \ll A^{-k} \left(A^2+(m+n)^{\frac 14}A+(mn)^{\frac 14}\right)(mnA)^\varepsilon.
\end{equation} 
These theorems together imply that
\begin{equation}
	\sqrt{mn} \sum_{A\leq r_j < 2A} \frac{|\rho_j(m)\rho_j(n)|}{\cosh\pi r_j} \\ \ll A^{-k} \min\left( (dd')^{\frac 16}(vw)^{\theta}A^2, A^2+(m+n)^{\frac 14}A+(mn)^{\frac 14} \right)(mnA)^\varepsilon.
\end{equation}
The following lemma gives an estimate for $\hat\varphi(r)$.

\begin{lemma} \label{lem:phi-hat-est}
If $r\geq 1$ then with $\varphi=\varphi_{a,x,T}$ as above we have
\begin{equation}
	\hat\varphi(r) \ll r^k \min\left(r^{-\frac 32},r^{-\frac 52}\frac xT\right).
\end{equation}
If $|r|\leq 1$ then $\hat\varphi(r)\ll |r|^{-2}$.
\end{lemma}

\begin{proof}
Recall that
\begin{equation}
\hat\varphi(r) = \frac{-i\,\xi_k(r)}{\cosh 2\pi r} \int_0^\infty \left( \cos(\tfrac{\pi k}2+\pi i r)J_{2ir}(x) - \cos(\tfrac{\pi k}2-\pi i r)J_{-2ir}(x) \right) \varphi(x) \mfrac{dx}{x},
\end{equation}
where $\xi_k(r) \asymp r^k$ as $r\to\infty$.
Sarnak and Tsimerman \cite[(47)--(48)]{sarnak-tsimerman} proved that
\begin{equation}
	e^{-\pi |r|} \int_{0}^\infty J_{2ir}(x) \varphi(x) \mfrac{dx}{x} \ll \min\left(|r|^{-\frac 32},|r|^{-\frac 52}\frac xT\right)
\end{equation}
for $|r|\geq 1$.
The first statement of the lemma follows.
The second is similar, using \cite[(43)]{sarnak-tsimerman}.
\end{proof}

Since $\min(x,y)\ll x^a y^{1-a}$ for any $a\in [0,1]$, we have
\begin{equation}
	\mathcal K^m(A) \ll \min\left(1,\frac{x}{AT}\right) \left(\sqrt A + (dd')^{\frac 1{12}}(vw)^{\frac{\theta}{2}}(m+n)^{\frac 18} + (dd')^{\frac 3{16}}(vw)^{\frac18+\frac 34\theta}\right)(mnA)^\varepsilon,
\end{equation}
where we used $a=\frac 12$ in the second term and $a=\frac 34$ in the third term.
Summing over $A$ we conclude that
\begin{equation} \label{eq:K-m-est}
	\mathcal K^m \ll  \left(\sqrt{\frac xT} + (dd')^{\frac 1{12}}(vw)^{\frac{\theta}{2}}(m+n)^{\frac 18} + (dd')^{\frac 3{16}}(vw)^{\frac18+\frac 34\theta} \right)(mnx)^\varepsilon.
\end{equation}

We turn to the estimate of the integral
\begin{equation}
	\mathcal K^e := \int_\R \pfrac{d}{d'}^{ir} \frac {L(\frac 12-2ir,\chi_{d'})L(\frac 12+2ir,\chi_{d})\fS_{d'}(v,2ir)\fS_{d}(w,2ir)}{|\zeta(1+4ir)|^2\cosh\pi r |\Gamma(\frac{k+1}{2}+ir)|^2}  \hat\varphi(r) \, dr.
\end{equation}
By symmetry it suffices to estimate the integrals $\mathcal K^e_0 = \int_0^1$ and $\mathcal K^e_1 = \int_1^\infty$.
Estimating the divisor sums trivially we find that
\begin{equation}
	|\fS_d(w,s)| \leq \sigma_0(w)^2.
\end{equation}
For $|r|\leq 1$ we have $|\zeta(1+4ir)|^2\gg r^{-2}$ and $\cosh\pi r |\Gamma(\frac{k+1}{2}+ir)|^2 \gg 1$, so by Lemma~\ref{lem:phi-hat-est} we have the estimate
\begin{equation} \label{eq:Ke0-est}
	\mathcal K^e_0 \ll (vw)^\varepsilon \int_0^1 \left|L(\tfrac 12-2ir,\chi_{d'}) L(\tfrac 12+2ir, \chi_d)\right| \, dr.
\end{equation}
Since $\cosh\pi r |\Gamma(\frac{k+1}{2}+ir)|^2 \sim \pi r^k$ for large $r$ and since $|\zeta(1+4ir)|^{-1}\ll r^{\varepsilon}$ for all $r$ we have by Lemma~\ref{lem:phi-hat-est} that
\begin{equation}
	\mathcal K^e_1 \ll (vw)^\varepsilon \int_1^\infty \big|L(\tfrac 12-2ir,\chi_{d'})L(\tfrac 12+2ir,\chi_d)\big| \mfrac{dr}{r^{3/2-\varepsilon}}.
\end{equation}
We multiply each Dirichlet $L$-function by $r^{-3/8}$ and the last factor by $r^{3/4}$, then apply H\"older's inequality in the case $\frac 16+\frac 16+\frac 23=1$.
We obtain
\begin{equation}\label{eq:Ke1-holder}
	\mathcal K^e_1 \ll (vw)^\varepsilon \left(\int_1^\infty |L(\tfrac 12+ir,\chi_{d'})|^6 \mfrac{dr}{r^{9/4}}\right)^{\frac 16} \left(\int_1^\infty |L(\tfrac 12+ir,\chi_d)|^6 \mfrac{dr}{r^{9/4}}\right)^{\frac 16}  \left(\int_1^\infty \mfrac{dr}{r^{9/8-\varepsilon}}\right)^{\frac 23}.
\end{equation}
Young \cite{young-subconvexity} proved that
\begin{equation}
	\int_{T}^{T+1} |L(\tfrac 12+ir,\chi_d)|^6 \, dr \ll (|d|(1+T))^{1+\varepsilon},
\end{equation}
from which it follows that $K_0^e \ll (vw)^\varepsilon |dd'|^{\frac 16+\varepsilon}$ and
\begin{equation}
	\int_1^\infty |L(\tfrac 12+ir,\chi_d)|^6 \mfrac{dr}{r^{9/4}} \leq \sum_{T=1}^\infty \mfrac{1}{T^{9/4}} \int_T^{T+1} |L(\tfrac 12+ir,\chi_d)|^6 \, dr \ll |d|^{1+\varepsilon}.
\end{equation}
This, together with \eqref{eq:Ke1-holder} proves that
\begin{equation} \label{eq:K-e-est}
	\mathcal K^e \ll (vw)^\varepsilon |dd'|^{\frac 16+\varepsilon}.
\end{equation}

Putting \eqref{eq:K-h-est}, \eqref{eq:K-m-est}, and \eqref{eq:K-e-est} together, we find that
\begin{multline}
	\sum_{4\mid c>0} \frac{S_k^+(m,n,c)}{c} \varphi\pfrac{4\pi\sqrt{mn}}{c} 
	\\* \ll 
	\left(\sqrt{\mfrac xT} + \mfrac{vw|dd'|^{\frac 23}}{x} + (dd')^{\frac 1{12}}(vw)^{\frac{\theta}{2}}(m+n)^{\frac 18} + (dd')^{\frac 3{16}}(vw)^{\frac18+\frac 34\theta} \right)(mnx)^\varepsilon.
\end{multline}
To unsmooth the sum of Kloosterman sums, we argue as in \cite{sarnak-tsimerman,aa-kloosterman} to obtain
\begin{equation}
	\sum_{4\mid c>0} \frac{S^+_k(m,n,c)}{c} \varphi\pfrac{4\pi\sqrt{mn}}{c} - \sum_{x\leq c< 2x} \frac{S^+_k(m,n,c)}{c} \ll \frac{T\log x}{\sqrt x} (mn)^\varepsilon.
\end{equation}
Choosing $T=x^{\frac 23}$ and using that $m+n\leq mn$ we obtain
\begin{equation} \label{eq:kloo-dyadic-est}
	\sum_{x\leq c< 2x} \frac{S^+_k(m,n,c)}{c} \ll \left(x^{\frac 16} +  \mfrac{vw|dd'|^{\frac 23}}{x} + (dd')^{\frac 5{24}}(vw)^{\frac 14+\frac\theta2} \right)(mnx)^\varepsilon.
\end{equation}
To prove \eqref{eq:main-kloo} we sum the inital segment $c\leq (dd')^a(vw)^b$ and apply the Weil bound \eqref{eq:weil-bound}, then sum the dyadic pieces for $c\geq (dd')^a(vw)^b$ using \eqref{eq:kloo-dyadic-est}.
To balance the resulting terms we take $a=\frac 49$ and $b=\frac 23$, which gives the bound
\begin{equation}
	\sum_{c\leq x} \frac{S^+_k(m,n,c)}{c} \ll \left(x^{\frac 16} + (dd')^{\frac 29}(vw)^{\frac 13}\right)(mnx)^\varepsilon.
\end{equation}
This completes the proof. \qed

\appendix

\section{Young's theorem for even discriminants}

\label{sec:appendix}

Let $D$ be a fundamental discriminant.
Then $|D|=q$ or $4q$, where $q$ is squarefree (but not necessarily odd).
For a positive even integer $k$, 
let $\mathcal B_k(q)$ denote the set of weight $k$ holomorphic Hecke newforms of level dividing $q$.
Our goal in this appendix is to prove the following generalization of Young's hybrid subconvexity result \cite{young-subconvexity}.

\begin{theorem}
Notation as above, we have
\begin{equation}
	\sum_{f\in B_k(q)} L(\tfrac 12, f\times \chi_D)^3 \ll (k|D|)^{1+\ep}.
\end{equation}
\end{theorem}

A corresponding generalization also holds for Maass cusp forms and Eisenstein series; for simplicity we only deal with the holomorphic case here.

For ease of comparison with \cite{young-subconvexity}, we have adopted the notation of that paper for this section only.
We will indicate the changes that need to be made and refer the reader to \cite{conrey-iwaniec} and \cite{young-subconvexity} for the remaining details.
Starting in Section~4 of \cite{young-subconvexity}, our goal is to show that
\begin{equation} \label{eq:smoothed-main-sum}
	\sum_{k\equiv a(4)} w\pfrac{k-1-2T}{\Delta} \sum_{f\in B_k(q)} \omega_f^* L(\tfrac 12, f\times \chi_D)^3 \ll \Delta (T|D|)^{1+\ep},
\end{equation}
where $w$ is a smooth nonnegative function with support in $[\frac 12,3]$ which equals $1$ on the interval $[1,2]$, and $a$ is determined by $i^k = \chi_D(-1)$.
Here $\omega_f^*$ is a Petersson weight satisfying $\omega_f^* \gg (kq)^{-\ep}$.
Applying the approximate functional equation and the Petersson formula as in Section~5 of \cite{young-subconvexity}, we find that it suffices to show the following.
\begin{proposition}\label{prop:young}
For $i=1,2,3$, 
let $w_i$ be a smooth weight function supported on $x\asymp N_i$, with $1\ll N_i\ll(qT)^{1+\ep}$ and with $w_i^{(k)}\ll N_i^{-k}$.
Then
\begin{multline}
	\sum_{n_1,n_2,n_3} w_1(n_1)w_2(n_2)w_3(n_3) \chi_D(n_1n_2n_3) \sum_{c\equiv 0(q)} \frac{S(n_1n_2,n_3;c)}{c} B\pfrac{4\pi\sqrt{n_1n_2n_3}}{c}  \\
	\ll (N_1N_2N_3)^{1/2} \Delta T (qT)^{\ep},
\end{multline}
where $S(m,n;c)$ is the ordinary Kloosterman sum,
\begin{equation}
	B(x) = B^{\text{holo}}(x) = \sum_{k\equiv a(4)} (k-1) w\pfrac{k-1-2T}{\Delta} J_{k-1}(x)
\end{equation}
and $J_{k-1}(x)$ is the $J$-Bessel function.
\end{proposition}

With $w_1$, $w_2$, and $w_3$ as in Proposition~\ref{prop:young}, let
\begin{equation}
	S(N_1,N_2,N_3;C;B) = \sum_{\substack{c\asymp C \\ c\equiv 0(q)}} \mathcal S(N_1,N_2,N_3;c),
\end{equation}
where
\begin{equation}
	\mathcal S(N_1,N_2,N_3;c) = \sum_{n_1,n_2,n_3} \chi_D(n_1n_2n_3) S(n_1n_2,n_3;c) w_1(n_1)w_2(n_2)w_3(n_3) B\pfrac{4\pi\sqrt{n_1n_2n_3}}{c}.
\end{equation}
We now follow Section~8 of \cite{young-subconvexity}, where the main difference is that we must keep track of the dependence on $\operatorname{lcm}(c,|D|)$, which we write as $cs$, with $s\in \{1,2,4,8\}$.
Applying Poisson summation modulo $c$ to the sum over the lattice $\Z^3$ we find that
\begin{equation}
	\mathcal S(N_1,N_2,N_3;c) = \sum_{m_1,m_2,m_3} G(m_1,m_2,m_3;c) K(m_1,m_2,m_3;c),
\end{equation}
where 
\begin{equation}
	G(m_1,m_2,m_3;c) = \frac 1{(cs)^3}\sum_{a_1,a_2,a_3\bmod c} \chi_D(a_1a_2a_3) S(a_1a_2,a_3;c) e\pfrac{a_1m_1+a_2m_2+a_3m_3}{cs}
\end{equation}
and
\begin{equation}
	K(m_1,m_2,m_3;c) = \int_{\R^3} w_1(t_1)w_2(t_2)w_3(t_3) B\pfrac{4\pi \sqrt{t_1t_2t_3}}{c} e\pfrac{-m_1t_1-m_2t_2-m_3t_3}{cs} dt_1 dt_2 dt_3.
\end{equation}
The analysis of the analytic piece $K(m_1,m_2,m_3;c)$ is almost exactly the same as in \cite[Section~8]{young-subconvexity}; simply replace $t_i$ by $t_i/s^{2/3}$ and apply Lemma~8.1.
The only difference is that the phase $e\pfrac{-m_1m_2m_3}{c}$ in (8.4) is replaced by
\begin{equation} \label{eq:phase}
	e\pfrac{-m_1m_2m_3}{s^3c}.
\end{equation}

For the remainder of this section we will focus on the arithmetic piece $G(m_1,m_2,m_3;c)$.
We begin by fixing notation.
Let $D=tq'$, where $t$ and $q'$ are fundamental discriminants with $t\mid 2^\infty$ and $q'$ odd, so that $\chi_D = \chi_t \chi_{q'}$.
With $q\mid c$ and $cs=\operatorname{lcm}(c,D)$ as before, we have $s=t/(c,t)$.
Finally, write $c=c_oc_e$, with $c_e\mid 2^\infty$ and $c_o$ odd.
Then $cs$ factors as $cs=c_0\cdot sc_e$ into odd and even parts.
From the twisted multiplicativity of the Kloosterman sums, a straightforward computation gives the factorization
\begin{equation} \label{eq:G-fac}
	G(m_1,m_2,m_3;c) = G(m_1,m_2,\bar c_o m_3;c_e) G(m_1,m_2,\bar c_e \bar s^3 m_3;c_o),
\end{equation}
where we choose the inverse $\bar c_o$ such that
\begin{equation}
	c_o\bar c_o\equiv 1\pmod{s^3 c_e}.
\end{equation}

The second term on the right-hand side of \eqref{eq:G-fac} was evaluated in Lemma~10.2 of \cite{conrey-iwaniec}, which we record here in the following lemma (see also (9.2) of \cite{young-subconvexity}).
Note that Young's definition of $G(m_1,m_2,m_3;c)$, which we are using here, is slightly different from that of Conrey-Iwaniec.
Let  $R_k(m) = S(0,m;k)$
denote the Ramanujan sum and let
\begin{equation}
	H(w;q) = \sum_{u,v(q)} \chi_q(uv(u+1)(v+1)) e\pfrac{(uv-1)w}{q}.
\end{equation}

\begin{lemma}\label{lem:ci-odd}
Let $c_o=qr$ with $c_o$ odd and $q$ squarefree. Suppose $m_1,m_2,m_3$ are integers with
\begin{equation} \label{eq:cong}
	(m_3,r)=1 \quad \text{ and } \quad (m_1m_2,q,r)=1.
\end{equation}
Then we have
\begin{equation}
	e\pfrac{-m_1m_2m_3}{c_o} G(m_1,m_2,m_3;c_o) =  \frac{\chi_{k\ell}(-1)h}{rq^2\varphi(k)} R_k(m_1) R_k(m_2) R_k(m_3) H(\overline{rhk}m_1m_2m_3;\ell),
\end{equation}
where $h=(r,q)$, $k=(m_1m_2m_3,q)$, and $\ell=q/hk$. 
If the coprimality conditions above are not satisfied, then $G(m_1,m_2,m_3;c_o)$ vanishes.
\end{lemma}

Petrow and Young \cite[Lemma~9.4]{petrow-young} evaluated $G(m_1,m_2,m_3;c_e)$ when $c_e$ is a power of $2$.

\begin{lemma}\label{lem:py-2}
Suppose that $c_e\mid 2^\infty$ and factor $m_i$ into even and odd parts as $m_i=m_i^em_i^o$.
Then
\begin{equation}
	e\pfrac{-m_1m_2m_3}{s^3c_e} G(m_1,m_2,m_3;c_e) = \frac{s^3c_e^2}{t} \sum_{\Delta\mid 64} \frac 1{\varphi(\Delta)} \sum_{\chi\bmod \Delta} \mathfrak g_\chi \chi(m_1^om_2^om_3^o),
\end{equation}
where $\mathfrak g_\chi$ depends on $m_1^e,m_2^e,m_3^e,t,c_e,\chi$ and is bounded by an absolute constant.
\end{lemma}

Note that the phase terms in Lemmas~\ref{lem:ci-odd} and \ref{lem:py-2} combine to give
\begin{equation}
	e\pfrac{m_1m_2m_3}{s^3c},
\end{equation}
which exactly matches the phase term \eqref{eq:phase} coming from $K(m_1,m_2,m_3;c)$.

The last result we require is the following analogue of Lemma~9.3 of \cite{young-subconvexity}.
The remainder of the proof of \ref{prop:young} follows the proof of Proposition~7.3 of \cite{young-subconvexity}.

\begin{lemma}
Let $c=q'r$ with $q'$ odd and squarefree.
Let $\alpha_{m_1}$, $\beta_{m_2}$, and $\gamma_{m_3}$ be sequences of complex numbers satisfying $\alpha_{m_1} = \alpha_{m_1^e}\alpha_{m_1^o}$, $\beta_{m_2} = \beta_{m_2^e}\beta_{m_2^o}$, $\gamma_{m_3} = \gamma_{m_3^e}\gamma_{m_3^o}$, and $|\alpha_{m_1^e}|=|\beta_{m_2^e}|=|\gamma_{m_3^e}|=1$, and let $\delta_r$ be an arbitrary sequence of complex numbers.
Then for $U\geq 1$ we have
\begin{multline} \label{eq:Lem93bound}
	\int_{|u|\leq U} \left| \sum_{\substack{m_1,m_2,m_3 \\ m_i \asymp M_i}} \sum_{r\asymp R} \alpha_{m_1} \beta_{m_2} \gamma_{m_3} \delta_r G(m_1,m_2,m_3;c) e\pfrac{-m_1m_2m_3}{s^3c} \left( \frac{m_1 m_2 m_3}{c} \right)^{iu}  \right| \, du \\
	\ll \frac{q^{\frac 12+\ep}}{Rq^2} \left(qU + M_1M_2\right)^{\frac 12} \left(qU + M_3R\right)^{\frac 12} \left( \sum_{d,m_1,m_2,m_3,r} d^{1+\ep} |\alpha_{m_1}\beta_{m_2}\gamma_{dm_3}\delta_{dr}|^2 \right)^{\frac 12}.
\end{multline}
\end{lemma}

Remark: as in \cite[Lemma~9.3]{young-subconvexity}, when $\gamma_{m_3},\delta_r\ll 1$ the sum over $d$ does not change the bound which arises from $d=1$.

\begin{proof}
Using Lemmas~\ref{lem:ci-odd} and \ref{lem:py-2}, the left-hand side of \eqref{eq:Lem93bound} is
\begin{multline}
	\ll \sum_{hk\ell=q'} \frac{h}{Rq^2\varphi(k)} \sum_{\Delta\mid 64} \frac{1}{\varphi(\Delta)} \int_{|u|\leq U} \Big|	\sideset{}{^*}\sum_{m_1,m_2,m_3} \sideset{}{^*}\sum_{r\asymp R} \alpha_{m_1} \beta_{m_2} \gamma_{m_3} \delta_r \mathfrak g_\chi \chi(m_1^o m_2^o m_3^o) \\
	\times R_k(m_1) R_k(m_2) R_k(m_3) H(\overline{rhk}m_1m_2m_3;\ell) \pfrac{m_1^om_2^om_3^o}{q'}^{iu} \Big|\,du,
\end{multline}
where the star indicates that the sum is restricted by the coprimality conditions \eqref{eq:cong}.
Using that $R_k(m_i) = R_k(m_i^e)R_k(m_i^o)$ and $|R_k(m)|\leq (k,m)$ we bound the above by
\begin{multline}
	\ll \sum_{hk\ell=q'} \frac{h}{Rq^2\varphi(k)} \sum_{\Delta\mid 64} \frac{1}{\varphi(\Delta)} \sum_{\substack{j_1,j_2,j_3 \\ j_i\ll \log_2(M_i)}} \int_{|u|\leq U} \Big| \sideset{}{^*}\sum_{\substack{m_1^o,m_2^o,m_3^o \\ m_i^o \asymp M_i/2^{j_i}, m_i^e=2^{j_i}}} \sideset{}{^*}\sum_{r\asymp R} \alpha_{m_1^o} \beta_{m_2^o} \gamma_{m_3^o} \delta_r \chi(m_1^o m_2^o m_3^o) \\
	\times R_k(m_1^o) R_k(m_2^o) R_k(m_3^o) H(\overline{rhk}bm_1^om_2^om_3^o;\ell) \pfrac{m_1^om_2^om_3^o}{q'}^{iu} \Big|\,du,
\end{multline}
where $b=m_1^em_2^em_3^e$.
Now following the proof of Lemma~9.3 of \cite{young-subconvexity} almost exactly, we obtain the desired bound.
\end{proof}

\bibliographystyle{plain}
\bibliography{dist-geom-inv-bib}

\end{document}